\documentclass[ppn]{svjour}
\usepackage{soul}
\usepackage[T1]{fontenc}
\usepackage[utf8]{inputenc}

\usepackage{amssymb,amsfonts,amsmath}
\usepackage[usenames,dvipsnames]{color}
\usepackage{dsfont}
\usepackage{tikz}
\usetikzlibrary{patterns}
\usepackage{bbm}
\usepackage[colorlinks=true]{hyperref}
\hypersetup{urlcolor=blue, citecolor=red, linkcolor=blue}
\newcounter{Hequation}

\makeatletter\g@addto@macro\equation{\stepcounter{Hequation}}\makeatother

\newcommand{\Email}[1]{\href{mailto:#1}{\textsf{#1}}}
\newcommand{\be}[1]{\begin{equation}\label{#1}}
\newcommand{\ee}{\end{equation}}
\renewcommand{\(}{\left(}
\renewcommand{\)}{\right)}

\newcommand{\step}[1]{\par\medskip\noindent\textit{\textbf{$\bullet$ #1}}\medskip}
\newcommand{\Qed}{\hfill\ \qed}

\newcommand{\R}{{\mathbb R}}

\newcommand{\rint}{\int_{\R^d}}
\newcommand{\riint}{\iint_{\R^d\times\R^d}}
\newcommand{\irdx}[1]{\int_{\R^d}{#1}\,\d x}
\newcommand{\irdv}[1]{\int_{\R^d}{#1}\,\d v}
\newcommand{\irdmu}[1]{\int_{\R^d}{#1}\,\d\mu}
\newcommand{\irdxv}[1]{\iint_{\R^d\times\R^d}{#1}\,\d x\,\d v}
\newcommand{\irdvv}[1]{\iint_{\R^d\times\R^d}{#1}\,\d v\,\d v'}
\newcommand{\irdxmu}[1]{\iint_{\R^d\times\R^d}{#1}\,\d x\,\d\mu}
\newcommand{\n}[1]{\|#1\|}
\newcommand{\nrm}[2]{\|{#1}\|_{#2}}
\newcommand{\Ff}{\widehat f}
\newcommand{\M}{F}
\renewcommand\d{\mathrm d}
\newcommand{\bangle}[1]{\left\langle #1\right\rangle}
\newcommand{\re}{\mathop{\mathrm{Re}}}
\newcommand{\dt}{\frac{\mathrm d}{\mathrm dt}}
\newcommand{\I}[1]{\mathsf I_{#1}}
\DeclareMathOperator{\pv}{pv}
\newcommand{\vertiii}[1]{{\left\vert\kern-0.25ex\left\vert\kern-0.25ex\left\vert #1\right\vert\kern-0.25ex\right\vert\kern-0.25ex\right\vert}}

\begin{document}
%%%%%%%%%%%%%%%%%%%%%%%%%%%%%%%%%%%%%%%%%%%%%%%%%%%%%%%%%%%%%%%%%%%%%%
%%%%%%%%%%%%%%%%%%%%%%%%%%%%%%%%%%%%%%%%%%%%%%%%%%%%%%%%%%%%%%%%%%%%%%
\title{Fractional hypocoercivity}
\titlerunning{Fractional hypocoercivity}

\author{Emeric Bouin, Jean Dolbeault, Laurent Lafleche}
\authorrunning{E.~Bouin, J.~Dolbeault, L.~Lafleche}

\institute{CEREMADE (CNRS UMR n$^\circ$ 7534), PSL University, Universit\'e Paris-Dauphine\newline
Place de Lattre de Tassigny, 75775 Paris 16, France.\newline 
E-mail: \Email{bouin@ceremade.dauphine.fr} (E.B.), \Email{dolbeaul@ceremade.dauphine.fr} (J.D.),\newline
\hspace*{30pt}\Email{lafleche@ceremade.dauphine.fr} (L.L.)}

\date{\today}
\maketitle
\thispagestyle{empty}

\begin{abstract} This paper is devoted to kinetic equations without confinement. We investigate the large time behaviour induced by collision operators with fat tailed local equilibria. Such operators have an anomalous diffusion limit. In the appropriate scaling, the macroscopic equation involves a fractional diffusion operator so that the optimal decay rate is determined by a fractional Nash type inequality. At kinetic level we develop an $\mathrm L^2$-hypocoercivity approach and establish a rate of decay compatible with the fractional diffusion limit. \end{abstract}

%%%%%%%%%%%%%%%%%%%%%%%%%%%%%%%%%%%%%%%%%%%%%%%%%%%%%%%%%%%%%%%%%%%%%%
\\[4pt]\noindent\emph{Mathematics Subject Classification (2020):\/} Primary: \href{https://mathscinet.ams.org/mathscinet/search/mscbrowse.html?sk=default&sk=82C40&submit=Chercher}{82C40}; Secondary: \href{https://mathscinet.ams.org/mathscinet/search/mscbrowse.html?sk=default&sk=76P05&submit=Chercher}{76P05}; 
\href{https://mathscinet.ams.org/mathscinet/search/mscbrowse.html?sk=default&sk=35K65&submit=Chercher}{35K65}; 
\href{https://mathscinet.ams.org/mathscinet/search/mscbrowse.html?sk=default&sk=35Q84&submit=Chercher}{35Q84}; 
\href{https://mathscinet.ams.org/mathscinet/search/mscbrowse.html?sk=default&sk=35P15&submit=Chercher}{35P15}.\bigskip
\section{Introduction and main results}\label{Sec:intro}

We study the decay rates of the solutions in the kinetic equation
\be{eq:main}
\partial_tf+v\cdot\nabla_xf=\mathsf Lf\,,\quad f(0,x,v)=f^{\mathrm{in}}(x,v)
\ee
when the \emph{local equilibrium}~$\M$ has a \emph{fat tail} given for some $\gamma>0$ by
\be{eq:hypfattailed}
\forall\,v\in\R^d,\quad\M(v)=\frac{c_\gamma}{\bangle v^{d+\gamma}}\quad\mbox{where}\quad\bangle v:=\sqrt{1+|v|^2}\,.
\ee
In~\eqref{eq:main}, the distribution function $f(t,x,v)$ depends on a \emph{position} variable $x\in\R^d$, on a \emph{velocity} variable $v\in\R^d$, and on \emph{time} $t\ge 0$. The \emph{collision operator} $\mathsf L$ acts only on the $v$ variable and, by assumption, its null space is spanned by~$\M$. In~\eqref{eq:hypfattailed} the normalization constant is $c_\gamma=\pi^{-d/2}\,\Gamma\big((d+\gamma)/2\big)/\Gamma(\gamma/2)$ and associated to the measure
\[
\d\mu=\M^{-1}(v)\,\d v\,,
\]
we define a scalar product and a norm respectively by
\be{ScalarProduct-Norm}
\bangle{f,g}:=\irdmu{\bar f\,g}\quad\mbox{and}\quad\n{f}^2:=\irdmu{|f|^2}
\ee
for functions $f$ and $g$ of the variable $v\in\R^d$. Here $\bar f$ denotes the complex conjugate of $f$, as we shall later allow for complex valued functions. For any $k\in\R$, we define
\[
\vertiii f_k:=\nrm f{\mathrm L^1(\d x\,\d v)\cap\mathrm L^2(\bangle v^k\,\d x\,\d\mu)}:=\(\nrm f{\mathrm L^1(\d x\,\d v)}^2+\nrm f{\mathrm L^2(\bangle v^k\,\d x\,\d\mu)}^2\)^{1/2}\,.
\]

We consider three examples of linear collision operators and define for each of them an associated parameter $\beta$, to be discussed later:\\
$\rhd$ the \emph{Fokker-Planck} operator with $\beta:=2$ and local equilibrium $\M$
\[
\mathsf L_1 f:=\nabla_v\cdot\(\M\,\nabla_v\big(\M^{-1}f\big)\).
\]
$\rhd$ the \emph{linear Boltzmann} operator, or \emph{scattering} collision operator
\[
\mathsf L_2 f:=\rint\mathrm b(\cdot,v')\,\Big(f(v')\,\M(\cdot)-f(\cdot)\,\M(v')\Big)\,\d v'\,,
\]
with positive, locally bounded \emph{collision frequency}
\be{hyp:b_beta}
\nu(v):=\rint\mathrm b(v,v')\,\M(v')\,\d v'\underset{|v|\to+\infty}\sim|v|^{-\beta}
\ee
for a given $\beta\in\R$, and \emph{local mass conservation}, that is,
\be{hyp:b_mass}
\rint\big(\mathrm b(v,v')-\mathrm b(v',v)\big)\,\M(v')\,\d v'=0\,.
\ee
We assume the existence of a constant $Z>0$ such that, for any $v,\,v'\in\R^d$,
\be{hyp:b_bounds}
\mathrm b(v,v')\ge\frac1Z\bangle v^{- \beta} \bangle{v'}^{-\beta}\,.
\ee
We also assume that for any $k \in (0,\gamma+\beta)$,
\begin{align}
&\mathcal C_{\mathrm b}(k):=\sup_{v\in\R^d}\bangle v^{\beta}\rint\mathrm b(v',v)\bangle{v'}^k\,\M(v')\,\d v'<+\infty\,,\label{hyp:b_bounds2}\\
&\mathcal C_{\mathrm b}:=\riint\frac{\mathrm b(v',v)^2}{\nu(v')\,\nu(v)}\M\M'\,\d v\,\d v'<+\infty\,.\label{hyp:b_bounds3}
\end{align}
All these assumptions are verified for instance by the \emph{collision kernel} $\mathrm b$ such that
\begin{align*}
&\mbox{either}&\mathrm b(v,v')&=Z^{-1} \bangle{v'}^{- \beta}\bangle v^{- \beta}&\mbox{with}\quad|\beta|\le\gamma\,,\\
&\mbox{or}&\mathrm b(v,v')&=|v'-v|^{- \beta}&\mbox{with}\quad \beta\in [0,d/2)\,.
\end{align*}
$\rhd$ the \emph{fractional Fokker-Planck} operator
\[
\mathsf L_3 f:=\Delta_v^{\sigma/2}f+\nabla_v\cdot\(E\,f\)
\]
with $0<\sigma<2$, $\beta:= \sigma - \gamma$ and a radial \emph{friction force} $E=E(v)$ as a solution of
\be{eq:def_E}
\mathsf L_3 \M=\Delta_v^{\sigma/2}\M+\nabla_v\cdot\(E\,\M\)=0\,.
\ee
It turns out from a technical result exposed in Appendix~\ref{Sec:Ebound} that such a friction force then behaves like $\bangle v^{-\beta} v$ at infinity. 

We shall say that \emph{Assumption ~{\rm (H)} holds} if $\mathsf L$ is one of the three operators corresponding to the cases~$\mathsf L=\mathsf L_1$,~$\mathsf L=\mathsf L_2$, or~$\mathsf L=\mathsf L_3$, with corresponding assumptions and parameter $\beta$. 

Observe that due to total mass conservation, an initial data with finite total mass will necessary go to zero as $t$ goes to infinity since there is no global equilibrium state with finite mass except from zero. The aim of this paper is thus to derive rates of decay to zero. We shall use the notation $\beta_+ = \max(0,\beta)$ and the convention $1/0_+ = +\infty$. 
Let
\be{alpha}
\alpha := \begin{cases}
\frac{\gamma + \beta}{1 +\beta}\in(0,2)\quad&\mbox{if}\quad \gamma<2 + \beta\,,\\
2\,\quad\quad&\mbox{if}\quad \gamma\ge2 + \beta\,.
\end{cases}
\ee
%---------------------------------------------------------------------
\begin{theorem}\label{th:main2} Let $d\ge2$, $\beta\in\R$, $\gamma> {\max\{0,- \beta\}}$, $\alpha$ given by~\eqref{alpha} and $k \in [0,\gamma)$. Under Assumption~{\rm (H)}, let us consider a solution $f$ to~\eqref{eq:main} with initial condition $f^{\mathrm{in}}\in\mathrm L^1(\d x\,\d v)\cap\mathrm L^2\big(\bangle v^k\d x\,\d\mu\big)$.
If $\gamma \neq 2 + \beta$ or if $\gamma=2 + \beta$ and $\frac{k}{\beta_+} >\frac d2$, then 
\[
\forall\,t\ge0\,,\quad  \nrm{f(t,\cdot,\cdot)}{\mathrm L^2(\d x\,\d\mu)}^2\le \frac C{(1+t)^\tau}\,\vertiii{f^{\mathrm{in}}}_k^2\quad\mbox{with}\quad\tau =\min\big\{\tfrac d{\alpha },\tfrac k{\beta_+}\big\}\,.
\]
In the critical case $\gamma=2 + \beta$, and with either $k=0$ if $\beta < 0$, or $k>0$ if $\beta\ge0$, and under the additional condition $\frac{k}{\beta_+}\le\frac d2$ if $d\ge3$, 
\[
\forall\,t\ge2\,, \quad \nrm{f(t,\cdot,\cdot)}{\mathrm L^2(\d x\,\d\mu)}^2\le \frac C{(t\,\log t)^{d/2}}\,\vertiii{f^{\mathrm{in}}}_k^2\,.
\]
In the above estimates, $C>0$ is a constant which does not depend on $f^{\mathrm{in}}$.
\end{theorem}
%---------------------------------------------------------------------
In Theorem~\ref{th:main2}, the case $\gamma>\beta+2$ gives rise to a decay rate corresponding to a standard, \emph{i.e.}, non-fractional diffusion regime, with $\alpha=2$, as we shall see later. For legibility, we state the $d=1$ case separately.
%---------------------------------------------------------------------
\begin{theorem}\label{th:main1} Let $d = 1$, $\beta\in\R$, ${\gamma>\max\{0,- \beta\}}$, $\alpha$ given by~\eqref{alpha} and $k\in [0,\gamma)$. Under Assumption~{\rm (H)}, let us consider a solution $f$ of~\eqref{eq:main} with initial condition $f^{\mathrm{in}}\in\mathrm L^1(\d x\,\d v)\cap\mathrm L^2\big(\bangle v^k\d x\,\d\mu\big)$. 
If $\gamma \neq 2 +\beta$, then the estimate
\[
\forall\,t\ge0\,,\quad  \nrm{f(t,\cdot,\cdot)}{\mathrm L^2(\d x\,\d\mu)}^2\le \frac C{(1+t)^\tau}\,\vertiii{f^{\mathrm{in}}}_k^2
\]
holds
\begin{itemize}
\item[$\bullet$] for any $\tau < \tfrac{k+\gamma}{k\,\alpha-\,\gamma+\,\beta\,(\alpha+1)}$ if $\beta > 1$, $\gamma\in(1,\beta)$, $k\in\(\tfrac{\gamma}{\alpha},\gamma\)$, 
\item[$\bullet$] with $\tau=\min\big\{\tfrac d{\alpha },\tfrac k{\beta_+}\big\}$ in the other cases.
\end{itemize}
In the case $\gamma=2 + \beta$, then
\[
\forall\,t\ge2\,, \quad \nrm{f(t,\cdot,\cdot)}{\mathrm L^2(\d x\,\d\mu)}^2\le \frac C{(t\,\log t)^{d/2}}\,\vertiii{f^{\mathrm{in}}}_k^2\,,
\]
with either $k=0$ if $\beta<0$, or $k >0$ if $\beta\ge0$.
\\
In the above estimates, $C>0$ is a constant which does not depend on $f^{\mathrm{in}}$.
\end{theorem}
%---------------------------------------------------------------------

See Figures~\ref{fig:d3},~\ref{fig:d2} and~\ref{fig:d1} for a representation of the regions of the parameters respectively in dimensions $d=3$, $d=2$ and $d=1$.

If $d\ge2$ and $\beta\ge0$, the threshold between the region with decay rate $O(t^{-k/ \beta})$, with $k<\gamma$ but close enough to $\gamma$, and the region with decay rate $O(t^{-d/\alpha })$ is obtained by solving $\frac d{\alpha}=\frac k\beta$ in the limit case $k=\gamma$. The corresponding curve is given by $\beta\mapsto\gamma_\star(\beta)$ defined as
\be{gammastar}\begin{array}{rl}
\textstyle\gamma_\star(\beta):=\max\left\{\frac 12\(\sqrt{(4\,d+1)\, \beta^2 + 4\,d\,\beta} - \beta\),\frac d2\, \beta\right\}\quad&\text{if}\quad d\ge3\,,\\[8pt]
\textstyle\gamma_\star(\beta)=\frac12\(\sqrt{\beta\,(9\,\beta+8)} - \beta\)\quad&\text{if}\quad d=2\,.
\end{array}\ee
If $d\ge3$, notice that $\gamma_\star(\beta) = \big(\sqrt{(4\,d+1)\, \beta^2 + 4\,d\,\beta} - \beta\big)/2$ if $\beta \in [0,4/(d-2)]$ and $\gamma_\star( \beta)=\frac d2\,\beta$ if $  \beta\ge 4/(d-2)$. See Figures~\ref{fig:d3} and~\ref{fig:d2}. If $d=1$, the results when $\gamma < 2 + \beta$ slightly differs from the case $d\ge2$ and
\[
\textstyle\gamma_\star(\beta)= \frac12\(\sqrt{(5\,\beta + 4)\, \beta}- \beta\) \quad \text{if}\quad \beta\in[0,1],
\]
while a new intermediate rate appears when $\beta \geq 1$ and $\gamma\in(1,\beta)$.

%.....................................................................
\begin{figure}
\begin{center}
\includegraphics[width=8cm]{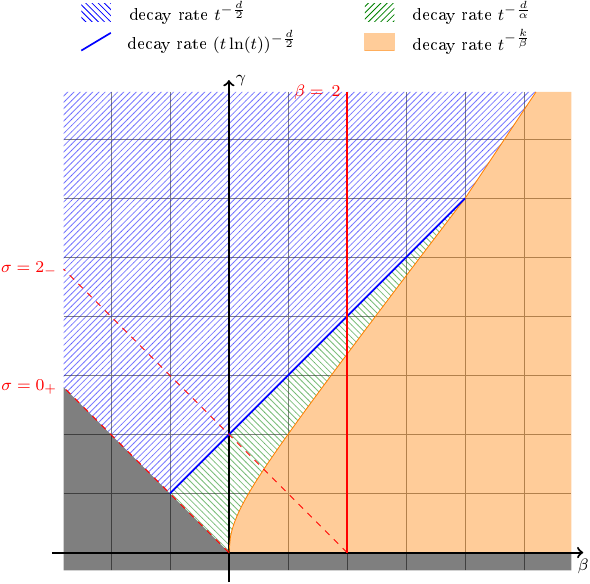}
\caption{As $t\to+\infty$, decay rates are at most $O(t^{-k/\beta})$ if $- \beta<0<k<\gamma$ sufficiently close to~$\gamma$ and $\gamma<\gamma_\star( \beta)$, with $\gamma_\star$ given by~\eqref{gammastar}, and otherwise either $O(t^{-d/\alpha})$ if $\max\{0,- \beta\}<\gamma<2 + \beta$ or $O(t^{-d/2})$ if $\gamma>\max\{2 + \beta,- \beta\}$. The picture corresponds to $d=3$. In Case~$\mathsf L=\mathsf L_3$, $\gamma$ is limited to the strip enclosed between the two dashed red lines.
}\label{fig:d3}
\end{center}
\end{figure}
%.....................................................................
\begin{figure}
\begin{center}
\includegraphics[width=8cm]{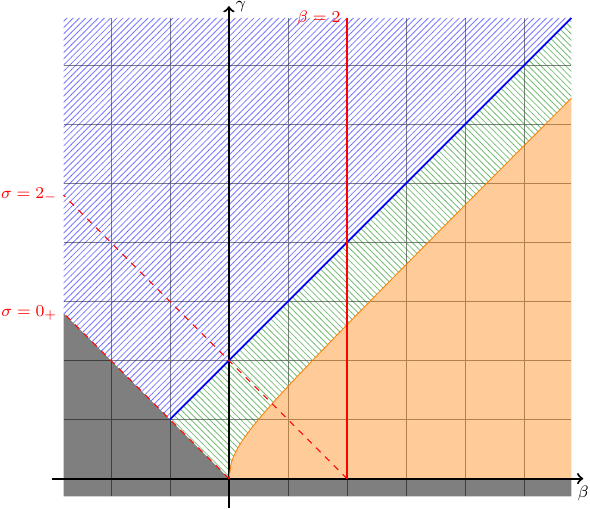}
\caption{Decay rates depending on $\beta$ and $\gamma$ in dimension $d=2$, as $t\to+\infty$. The caption convention is the same as for Figure~\ref{fig:d3}. When $ \beta\ge0$, the upper threshold of the region with decay rate $O(t^{-k/\beta})$, with $k$ close enough to $\gamma$, is $\gamma=\gamma_\star(\beta)$.}\label{fig:d2}
\end{center}
\end{figure}
%.....................................................................
\begin{figure}
\begin{center}
\includegraphics[width=8cm]{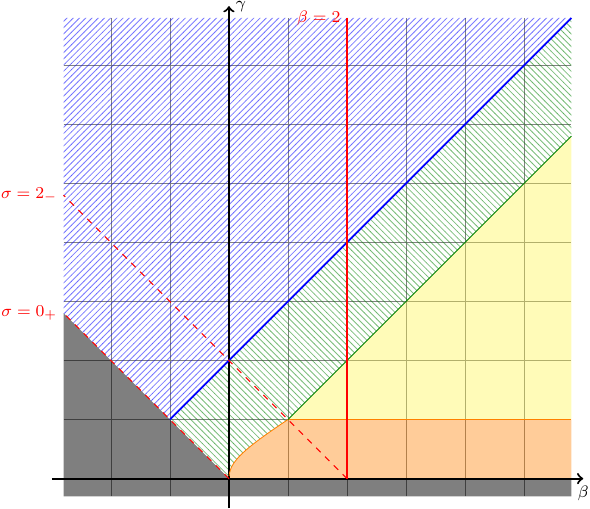}
\caption{Decay rates depending on $\beta$ and $\gamma$ in dimension $d=1$, as $t\to+\infty$. When $\beta\ge0$, $k$ is chosen arbitrarily close to~$\gamma$. The caption convention is the same as for Figure~\ref{fig:d3} except for $1<\gamma<\beta$ which corresponds to the intermediate decay rate $\tau < \frac{2\gamma}{\gamma\(\alpha-1\)+\beta\(\alpha+1\)}$ indicated in Theorem~\ref{th:main2}.}\label{fig:d1}
\end{center}
\end{figure}
%\clearpage
%.....................................................................

The large time decay rates are governed by the scaling properties of~\eqref{eq:main}. According to~\cite{mellet_fractional_2011,MR2588245} (more references will be given later), a \emph{local equilibrium with fat tail} implies that the diffusion limit involves a \emph{fractional diffusion} operator. As in~\cite{bouin_hypocoercivity_2017}, the key idea is that the (fractional) diffusion limit determines the rate of decay. Let us explain how the exponent $\alpha$ arises through a formal analysis in the case of the simple scattering operator $\mathsf L=\mathsf L_2$ corresponding~to
\[
\mathrm b(v,v')=Z^{-1}\bangle v^{-\beta} \bangle{v'}^{-\beta}\,,\quad Z:=\irdv{\bangle v^{-\beta} \M(v)}\,,\quad \beta\in\R\,.
\]
Let us investigate the diffusion limit as $\varepsilon\to 0^+$ of the scaled kinetic equation written in terms of the Fourier variable
$\xi$ corresponding to the macroscopic position variable $\varepsilon\,x$ :
\be{Eqn1}
\varepsilon^\alpha\,\partial_t\Ff+i\,\varepsilon\,v\cdot\xi\,\Ff=\mathsf L\Ff\,.
\ee
The exponent $\alpha$, which determines the macroscopic time scale, is to be chosen. By the \emph{local mass conservation} property, $\irdv{\mathsf Lf}=0$, the spatial density, defined as
\be{SpatialDensity}
\rho_f(t,x):=\irdv{f(t,x,v)}\,,
\ee
solves in Fourier variables the continuity equation
\be{cont-equ}
\partial_t\widehat\rho_f+i\,\varepsilon^{1-\alpha}\irdv{v\cdot\xi\,\Ff}=0\,.
\ee
The \emph{fractional diffusion limit} $\varepsilon\to 0^+$ can be obtained by a formal Hilbert expansion as in~\cite{Puel_2011}, in which only the case $\beta=0$ is covered, or as in~\cite{MR3535408}, where the collision frequency is $|v|^{- \beta}$. Here we use a more direct computation.

Rewriting the scattering operator as
\[
\mathsf L\Ff=\bangle v^{- \beta}\(\mathsf r_f\,\frac{\M}{Z}-\Ff\)\quad\mbox{with}\quad\mathsf r_f(t,\xi):=\irdv{\bangle {v'}^{- \beta}\,\Ff(t,\xi,v')}'\,,
\]
the kinetic equation \eqref{Eqn1} takes the form
\[
\left( \bangle v^{- \beta}+i\,\varepsilon\,v\cdot\xi\right)\Ff=\bangle v^{- \beta}\,\mathsf r_f\,\frac{\M}{Z} - \varepsilon^\alpha\,\partial_t \Ff\,.
\]
The flux term in the continuity equation \eqref{cont-equ} can be rewritten as
\be{flux}
i\,\varepsilon^{1-\alpha}\irdv{v\cdot\xi\,\Ff}=\mathsf b_\varepsilon \frac{\mathsf r_f}{Z}- i\,\varepsilon\irdv{\frac{v\cdot \xi\,\partial_t \Ff}{\bangle v^{- \beta}+i\,\varepsilon\,v\cdot\xi}}\,,
\ee
with
\be{b-eps}
\mathsf b_\varepsilon:=i\,\varepsilon^{1-\alpha}\irdv{\frac{\bangle v^{- \beta}\,v\cdot\xi\,\M}{\bangle v^{- \beta}+i\,\varepsilon\,v\cdot\xi}}=\varepsilon^{2-\alpha}\irdv{\frac{\bangle v^{- \beta}\,(v\cdot\xi)^2\,\M}{\bangle v^{-2 \,\beta}+\varepsilon^2\,(v\cdot\xi)^2}}\,.
\ee
The second representation is due to the evenness of $\M$. Recalling \eqref{eq:hypfattailed}, we observe that the formal limit of the integral is finite for $\gamma>2 + \beta$ and, by rotational symmetry, of the form $\kappa\,|\xi|^2$, for some $\kappa>0$. In this case the appropriate macroscopic time scale is diffusive, \emph{i.e.}, $\alpha=2$, and the macroscopic limit is the heat equation for the limiting spatial density $\rho_0$, written in Fourier variables,
\[
\partial_t\widehat\rho_0+\kappa\,|\xi|^2\,\widehat\rho_0=0\,.
\]
Here we use that formally $\lim_{\varepsilon\to 0} f=\rho_0\,\M$, and therefore $\lim_{\varepsilon\to 0} \mathsf{r}_f=Z\,\widehat\rho_0$, and assume that the last term in \eqref{flux} is a perturbation, which tends to zero.

Now let us consider the case $\gamma<2+ \beta$. The computation of the asymptotic behaviour of $\mathsf b_\varepsilon$ is a bit tedious in this case. First note that in this case \hbox{$\mathsf b_\varepsilon<\infty$} requires $\min\{1,\gamma\} + \beta >0$, which we assume in the following. With the coordinate change $v=(\varepsilon\,|\xi|)^{-1/(1+\beta)}\,w$, we obtain
\[
\mathsf b_\varepsilon \approx \varepsilon^{\frac{\gamma + \beta}{1+\beta}-\alpha}\,|\xi|^{\frac{\gamma+\beta}{1+\beta}}\,\kappa\,,\quad \kappa:=c_\gamma\int_{\R^d} \frac{|w|^{- \beta-\gamma-d}\,(w\cdot\mathsf{e})^2}{|w|^{-2\,\beta}+(w\cdot\mathsf{e})^2}\,\d w
\]
as $\varepsilon\to 0^+$, with $\mathsf e=\xi/|\xi|$. By rotational symmetry, $\kappa$ is independent from $\mathsf e\in\mathbb S^{d-1}$. The appropriate choice of the macroscopic time scale is now such that~$\mathsf b_\varepsilon$ has a finite positive limit, which determines $\alpha$ as in~\eqref{alpha}. Our assumptions on $ \beta$ and $\gamma$ imply $0<\alpha<2$. The macroscopic limit is the \emph{fractional heat equation}
\be{fhe}
\partial_t\widehat\rho_0+\kappa\,|\xi|^\alpha\,\widehat\rho_0=0\,.
\ee

In the general case, $\alpha \leq 2$ covers the two cases, $\gamma\ge2 + \beta$ and $\gamma<2 + \beta$, with a standard macroscopic limit when $\alpha =2$, and a fractional diffusion limit when $\alpha <2$. If $\rho$ solves~\eqref{fhe}, then
\[
\frac d{dt}\nrm{\,\widehat\rho\,}{\mathrm L^2(\d x)}^2=-\,2\,\kappa\,\nrm{\,|\xi|^\frac{\alpha }2\,\widehat\rho\,}{\mathrm L^2(\d\xi)}^2\,.
\]
Using the \emph{fractional Nash inequality}
\be{eq:NashFrac}
\nrm\rho{\mathrm L^2(\d x)}\le\mathcal C_{\mbox{Nash}}\,\nrm\rho{\mathrm L^1(\d x)}^\frac{\alpha }{d+\alpha }\,\nrm{\,|\xi|^\frac{\alpha }2\,\widehat\rho\,}{\mathrm L^2(\d\xi)}^\frac d{d+\alpha }\,,
\ee
and Plancherel's identity, we obtain $\nrm{\rho(t,\cdot)}{\mathrm L^2(\d x)}=O\big(t^{-d/\alpha }\big)$ as $t\to+\infty$. The proof of~\eqref{eq:NashFrac} can be found in~\cite{Nash58} if $\alpha =2$ and the extension to the case $\alpha <2$ is straightforward. This heuristic analysis is responsible for the rates $\tau=d/\alpha $ of the solution of~\eqref{eq:main}, at least under appropriate $\bangle v{}\!^k$ moment conditions. At formal level, the decay estimates of the solution of~\eqref{Eqn1} are uniform as $\varepsilon\to0$. See Section~\ref{Sec:difflim} for additional comments.

\medskip For $\mathsf L=\mathsf L_1$ and $\mathsf L=\mathsf L_3$, we want to keep the same value for $\alpha$ as for $\mathsf L=\mathsf L_2$. This implicitly defines $\beta$. Technically, what matters is the \emph{Lyapunov function}
property, namely the fact that there exist three positive constant $a$, $b$ and $R$, a real parameter $\beta$, and a (smooth) positive \emph{Lyapunov function} $F=F(v)$ on $\R^d$ such that
\[
-\mathsf L F\le \(a\,\mathbbm 1_{B_R}-b\bangle v^{- \beta}\)F
\]
where $\mathsf L$ is a self-adjoint operator associated with a Dirichlet integral on the space of square integrable functions on $\R^d$ with respect to some probability measure. Details can be found for instance  in~\cite{bakry_rate_2008}. In integral form, this Lyapunov function property is used in Lemma~\ref{lem:Lyapunov}, responsible for the interpolation inequality of Proposition~\ref{prop:WPII}, and finally for the decay rates of Theorems~\ref{th:main2} and~\ref{th:main1}.

The expression of $\alpha$ can also be related with the macroscopic diffusion limit for $\mathsf L=\mathsf L_1$ and $\mathsf L=\mathsf L_3$. This is not as simple as for $\mathsf L=\mathsf L_2$ and we shall omit this computation here, even at formal level. The interested reader is invited to refer to~\cite{bouin2020quantitative} for such a justification.

\medskip Let us conclude this introduction by a brief review of the literature. \emph{Fractional diffusion limits} of kinetic equations attracted a considerable interest in the recent years. The microscopic jump processes are indeed easy to encode in kinetic equations and the diffusion limit provides a simple procedure to justify the use of fractional operators at macroscopic level. Formal derivations are known for a long time, see for instance~\cite{SCALAS_2003}, but rigorous proofs are more recent. In the case of linear scattering operators like those of Case~$\mathsf L=\mathsf L_2$, we refer to~\cite{mellet_fractional_2011,mellet_fractional_2010,Puel_2011,ben_abdallah_anomalous_2011} for some early results and to~\cite{MR2588245} for a closely related work on Markov chains. Numerical schemes which are asymptotically preserving have been obtained in~\cite{MR3470742,MR3535408}. Beyond the classical paper~\cite{degond_diffusion_2000}, we also refer to~\cite{mellet_fractional_2011,mellet_fractional_2010,Puel_2011,ben_abdallah_anomalous_2011} for a discussion of earlier results on standard, \emph{i.e.}, non-fractional, diffusion limits. Concerning the generalized Fokker-Planck operators of Case~$\mathsf L=\mathsf L_1$, such that local equilibria have fat tails, the problem has recently been studied in~\cite{MR3922536} in dimension $d=1$ by spectral methods and, from a probabilistic point of view, in~\cite{fournier_one_2018}. Depending on the range of the exponents, various regimes corresponding to Brownian processes, stable processes or integrated symmetric Bessel processes are obtained and described in~\cite{fournier_one_2018} as well as the threshold cases (some were already known, see for instance~\cite{cattiaux2019diffusion}). Higher dimensional results have recently been obtained in~\cite{fournier2018anomalous}. Concerning Case~$\mathsf L=\mathsf L_3$, the fractional diffusion limit of the fractional Vlasov-Fokker-Planck equation, or Vlasov–L\'evy–Fokker–Planck equation, has been studied in~\cite{cesbron_anomalous_2012,Aceves_Sanchez_2019,2019arXiv191111535A} when the friction force is proportional to the velocity. Here our Case~$\mathsf L=\mathsf L_3$ is slightly different, as we pick forces giving rise to collision frequencies of the order of $|v|^{- \beta}$ as $|v|\to+\infty$. We refer to~\cite{bouin2020quantitative} for new results, a recent overview and further references.

In the \emph{homogeneous case}, that is, when there is no $x$-dependence, it is classical to introduce a function $\Phi(v)=-\log\M(v)$, where $\M$ denotes the local equilibrium but is not necessarily of the form~\eqref{eq:hypfattailed}, and classify the possible behaviors of the solution~$f$ to~\eqref{eq:main} according to the growth rate of $\Phi$. Assume that the collision operator is either the generalized Fokker-Planck operator of Case~$\mathsf L=\mathsf L_1$ or the scattering operator of Case~$\mathsf L=\mathsf L_2$. Schematically, if
\[
\Phi(v)=\bangle v^\zeta\,,
\]
we obtain that $\nrm{f(t,\cdot)-M\,\M}{\mathrm L^2(\d\mu)}$ decays exponentially if $\zeta\ge1$, with $M=\irdv f$. In the range $\zeta\in(0,1)$, the Poincar\'e inequality of Case~$\mathsf L=\mathsf L_1$ has to be replaced by a \emph{weak Poincar\'e} or a \emph{weighted Poincar\'e inequality}: see~\cite{rockner_weak_2001,kavian_fokker-planck_2015,BDLS} and rates of convergence are typically algebraic in $t$. Summarizing, the lowest is the rate of growth of $\Phi$ as $|v|\to+\infty$, the slowest is the rate of convergence of~$f$ to~$M\,\M$. Now let us focus on the limiting case as $\zeta\to 0^+$. The turning point precisely occurs for the minimal growth which guarantees that $\M$ is integrable, at least for solutions of the homogeneous equation with initial data in $\mathrm L^1(\d v)$. Hence, if we consider
\[
\Phi(v)=\eta\,\log\bangle v,
\]
with $\eta<d$, then diffusive effects win over confinement and the unique local equilibrium with finite mass is $0$. To measure the sharp rate of decay of~$f$ towards~$0$, one can replace the Poincar\'e inequality and the weak Poincar\'e or the weighted Poincar\'e inequalities by \emph{weighted Nash inequalities}. See~\cite{bouin_diffusion_2019} for details. In this paper, we consider the case $\eta = d + \gamma>d$, which guarantees that~$\M$ is integrable. Standard diffusion limits can be invoked if $\gamma>2+\beta$, but here we are also interested in the regime corresponding to fractional diffusion limits, with $\gamma\le2+\beta$.

As explained in Section~\ref{Sec:intro}, standard diffusion limits provide an interesting insight into the \emph{micro/macro decomposition} which is the key of the \emph{$\mathrm L^2$-hypocoercive approach} of~\cite{dolbeault_hypocoercivity_2015}. Another parameter can be taken into account: the confinement in the spatial variable $x$. In presence of a confining potential $V=V(x)$ with sufficient growth and when~$\M$ has fast decay, typically for $\zeta\ge1$, the rate of convergence is found to be exponential. A milder growth of $V$ gives a slower convergence rate as analyzed in~\cite{cao_kinetic_2018}. If $e^{-V}$ is not integrable, the diffusion wins in the hypocoercive picture, and the rate of convergence of a finite mass solution of~\eqref{eq:main} towards $0$ can be captured by Nash and related Caffarelli-Kohn-Nirenberg inequalities: see~\cite{bouin_hypocoercivity_2017,bouin_diffusion_2019}.

A typical regime for fractional diffusion limits is given by local equilibria with fat tails which behave according to~\eqref{eq:hypfattailed} with $\gamma\in(0,2+\beta)$: $\M$ is integrable but has no standard diffusion limit. Whenever fractional diffusion limits can be obtained, it was expected that rates of convergence can also be obtained by an adapted $\mathrm L^2$-hypocoercive approach. In this paper, we shall consider only the case $V=0$ and measure the decay rate. In view of~\cite{lafleche_fractional_2020} (also see references therein), it is natural to expect that a fractional Nash type approach has to play the central role, and this is indeed what happens. The mode-by-mode hypocoercivity estimate shows that rates are of the order of $|\xi|^\alpha$ as $\xi\to0$ which results in the expected time decay. In this direction, let us mention that the spectral information associated with $|\xi|^\alpha$ is very natural in connection with the fractional heat equation as was recently observed in~\cite{ben-artzi_weak_2018}. As far as we know, asymptotic rates for~\eqref{eq:main} have not been studied so far, to the exception of the very recent results of~\cite{2019arXiv191111535A} which deal with the Vlasov–L\'evy–Fokker–Planck equation in the case of a spatial variable in the flat torus $\mathbb T^d$ by an $\mathrm H^1$-hypocoercivity method and the simplest version ($\beta=0$) of the scattering collision operator: see Section~\ref{Sec:torus} for more details. A preliminary version of the present paper can be found in~\cite{LaflechePhD2019}.

%%%%%%%%%%%%%%%%%%%%%%%%%%%%%%%%%%%%%%%%%%%%%%%%%%%%%%%%%%%%%%%%%%%%%%
%%%%%%%%%%%%%%%%%%%%%%%%%%%%%%%%%%%%%%%%%%%%%%%%%%%%%%%%%%%%%%%%%%%%%%
\section{Outline of the method}\label{Sec:Hypocoercivity}

%%%%%%%%%%%%%%%%%%%%%%%%%%%%%%%%%%%%%%%%%%%%%%%%%%%%%%%%%%%%%%%%%%%%%%
\subsection{Decay rates of the homogeneous solution}

If $f$ is an \emph{homogeneous solution} of~\eqref{eq:main}, that is, a function independent from $x\in\R^d$, with initial datum $f^{\mathrm{in}}\in\mathrm L^1_+(\d v)\cap\mathrm L^2(\d\mu)$ such that $\irdv{f^{\mathrm{in}}}=1$, then
\[
\dt\,\n{f-\M}^2=2\bangle{f,\mathsf Lf}\,.
\]
It is natural to ask whether such an estimate proves the convergence of the solution $f(t,\cdot)$ to $\M$ as $t\to+\infty$ and provides a rate of convergence. Let us assume that $\mathsf L$ is a self-adjoint operator on $\mathrm L^2(\d\mu)$ such that, for some $k\in(0,\gamma)$,\\[2pt]
(i) the \emph{interpolation inequality}
\be{GeneralInterpolation}
\irdmu{|g|^2}\le\mathcal C\,\Big(-\bangle{g,\mathsf Lg}\Big)^\frac1{1+a}\(\irdmu{|g|^2\bangle v^k}\)^\frac a{1+a}
\ee
holds for some $a>0$ and $\mathcal C>0$, if $\irdmu g=0$, \\[2pt]
(ii) there is a constant $\mathcal C_k$ such that
\[
\forall\,t\ge0\,,\quad\irdmu{|f(t,\cdot)|^2\bangle v^k}\le\mathcal C_k\irdmu{\left|f^{\mathrm{in}}\right|^2\bangle v^k} \,.
\]
Then\, an elementary computation shows the \emph{algebraic decay rate}
\[
\forall\,t\ge0\,,\quad\n{f(t,\cdot)-\M}^2\le\(\n{f^{\mathrm{in}}-\M}^{-2\,a}+\kappa\,a\,t\)^{-1/a}
\]
with $\kappa=2\,\mathcal C^{-(1+a)}\,\big(\mathcal C_k\irdmu{|f^{\mathrm{in}}|^2\bangle v^k}\big)^{-a}$. Note that the exponent $a$ depends on $k$.
This result is an indication that also in the general spatially non-homo\-ge\-neous case of~\eqref{eq:main}, we cannot expect a better rate of convergence. The main difficulty there is to understand the interplay of the transport operator $v\cdot\nabla_x$ and of the collision operator $\mathsf L$, which is the main issue of this paper.

%%%%%%%%%%%%%%%%%%%%%%%%%%%%%%%%%%%%%%%%%%%%%%%%%%%%%%%%%%%%%%%%%%%%%%
\subsection{Non-homogeneous solutions: mode-by-mode analysis}

Let us consider the measure $\d\mu=\M^{-1}(v)\,\d v$ and the Fourier transform of $f$ in $x$ defined by
\[
\Ff(t,\xi,v):= \frac{1}{(2\pi)^{d/2}}\rint e^{-\,i\,x\cdot\xi}f(t,x,v)\,\d x\,.
\]
If $f$ solves~\eqref{eq:main}, then the equation satisfied by $\Ff$ is
\[
\partial_t\Ff+\mathsf T\Ff=\mathsf L\Ff\,,\quad\Ff(0,\xi,v)=\Ff^{\mathrm{in}}(\xi,v)
\]
where $\mathsf T$ is the \emph{transport operator} in Fourier variables given by
\[
\mathsf T\Ff=i\,v\cdot\xi\,\Ff\,,
\]
and $\xi\in\R^d$ can be seen as a parameter, so that for each Fourier mode $\xi$, $\mathsf T$ is a multiplication operator and we can study the decay of $t\mapsto\Ff(t,\xi,\cdot)$. For this reason, we call it a \emph{mode-by-mode analysis}, as in~\cite{bouin_hypocoercivity_2017}.

For any given $\xi\in\R^d$, taken as a parameter, we consider $(t,v)\mapsto\Ff(t,\xi,v)$ on the complex valued Hilbert space $\mathrm L^2(\d\mu)$ with scalar product and norm given by~\eqref{ScalarProduct-Norm}. We define the orthogonal projection $\mathsf\Pi$ on the subspace generated by~$\M$~by
\[
\mathsf\Pi\,f=\rho_f\,\M\,,
\]
where $\rho_f$ is given by~\eqref{SpatialDensity} and observe that the property
\[
\mathsf\Pi\mathsf T\mathsf\Pi=0
\]
holds as a consequence of the radial symmetry of $\M$. Let us define the operator
\[
\mathsf A_\xi:=\frac1{\bangle v^2}\,\mathsf\Pi\,\frac{\(-\,i\,v\cdot\xi\)\bangle v^{\beta}}{1+\bangle v^{2\,|1+\beta|} |\xi|^2}\,,
\]
and the \emph{entropy functional} by
\[
\mathsf H_\xi[f]:=\n\Ff^2+\delta\,\re\bangle{\mathsf A_\xi\Ff,\Ff}\,.
\]
The definition of $\mathsf A_\xi$ is reminiscent of the computation of $\mathsf b_\varepsilon$ in~\eqref{b-eps}, with $\varepsilon=1$. In the case $\mathsf L=\mathsf L_2$ and $\beta\ge-1$, we can indeed notice that
\[
\mathsf A_\xi\mathsf{T\,\Pi}=\frac1{\bangle v^2}\,\mathsf\Pi\,\frac{\(v\cdot\xi\)^2 \bangle v^{\beta}}{1+\bangle v^{2\,|1+\beta|} |\xi|^2}\,\mathsf\Pi=\frac1{\bangle v^2}\,\mathsf\Pi\,\frac{\bangle v^{- \beta}\(v\cdot\xi\)^2}{\bangle v^{-2\, \beta}+\bangle v^2|\xi|^2}\,\mathsf\Pi\,.
\]
Compared with the expression $\mathsf b_1$, there are two minor differences: the $\bangle v^{-2}$ factor is needed for technical reasons, in order to compute $\bangle v$ moments and in particular $\I1$ and $\I2$ in Section~\ref{Sec:macro}; in the denominator of the symbol, $\bangle v^{-2\, \beta}+(v\cdot\xi)^2$ is replaced by $\bangle v^{-2\, \beta}+\bangle v^2|\xi|^2$ which has similar scaling properties as \hbox{$|v|\to+\infty$} but offers simpler integration properties. Moreover, the same definition for $\mathsf A_\xi$ can be used in the cases $\mathsf L=\mathsf L_1$ and $\mathsf L=\mathsf L_3$. The first elementary result is the observation that $\mathsf A_\xi$ is a bounded operator and that $\mathsf H_\xi[f]$ is equivalent to $\n f^2$ if $\delta>0$ is not too large.

%---------------------------------------------------------------------
\begin{lemma}\label{Lem:Equivalence} With the above notation, for any $\delta\in(0,2)$ and $f\in\mathrm L^2(\d\mu)$, we have
\[
|\bangle{\mathsf A_\xi f,f}|\le\frac12\,\n f^2\quad\mbox{and}\quad\(2-\delta\)\n f^2\le2\,\mathsf H_\xi[f]\le\(2+\delta\)\n f^2\,.
\]
\end{lemma}
%---------------------------------------------------------------------
We shall use the notation
\[
\varphi(\xi,v):=\frac{\bangle v^{\beta}}{1+\bangle v^{2\,|1+\beta|} |\xi|^2}\quad\mbox{and}\quad\psi(v):=\bangle v^{-2},
\]
and may notice that $\mathsf A_\xi\Ff=\psi\,\mathsf\Pi\mathsf T^*\,\varphi\,\Ff$, where $\mathsf T^*$ denotes the dual of $\mathsf T$ acting on $\mathrm L^2(\d\mu)$.
\begin{proof}[Proof of Lemma~\ref{Lem:Equivalence}]
With these definitions, we obtain $|\psi|\le1$ and $\left|(v\cdot\xi)\,\varphi(\xi,v)\right|\le1/2$, so that the Cauchy-Schwarz inequality yields
\[
\left|\bangle{\mathsf A_\xi f,f}\right|^2\le\rint|\psi(v)|^2\,|f(\xi,v)|^2\,\d v\rint|(v\cdot\xi)\,\varphi(\xi,v)|^2\,|f(\xi,v)|^2\,\d v\le\frac14\,\n f^4\,,
\]
which completes the proof of Lemma~\ref{Lem:Equivalence}.
\Qed\end{proof}

We observe that if $f$ solves~\eqref{eq:main}, then
\[
\dt\mathsf H_\xi[\Ff]= -\mathsf D_\xi[\Ff]:=\,2\,\langle\mathsf L\Ff,\Ff\rangle - \delta\,\mathsf R_\xi[\Ff]
\]
where $\mathsf R_\xi[\Ff]=-\,\dt\re\,\langle\mathsf A_\xi\Ff,\Ff\rangle$. Our goal is to relate $\mathsf H_\xi[\Ff]$ and $\mathsf D_\xi[\Ff]$. Any decay rate of $\mathsf H_\xi[\Ff]$ obtained by a Gr\"onwall estimate gives us a decay rate for $\n f^2$ by Lemma~\ref{Lem:Equivalence} and, using an inverse Fourier transform, in $\mathrm L^2(\d x\,\d\mu)$.

\medskip More notation will be needed. Let us define the weighted norms
\[
\nrm gk^2:=\rint|g|^2\bangle v^k\d\mu\,,
\]
so that in particular $\n g=\nrm g0$. A crucial observation, which will be used repeatedly, is the fact that for any constant $\kappa>0$,
\[
\nrm{g-\kappa\,\M}k^2=\nrm gk^2+\kappa^2\rint\bangle v^k\M\,\d v-2\,\kappa\rint\bangle v^kg\,\d v\ge\nrm{(1-\mathsf\Pi_k)\,g}k^2
\]
where
\[
\mathsf\Pi_k\,g:=\frac{\rint\bangle v^kg\,\d v}{\rint\bangle v^k\M\,\d v}\,\M\,.
\]
This is easily shown by optimizing the left-hand side of the inequality on $\kappa\in\R$. Notice that $\mathsf\Pi_0=\mathsf\Pi$.

The parameters $\beta$ and $\gamma$ are chosen as in Theorems~\ref{th:main2} and~\ref{th:main1} while $\alpha$ is given by~\eqref{alpha}: $\alpha =\frac{\gamma+\beta}{1+\beta}$ if $\gamma < 2 + \beta$ and $\alpha = 2$ if $\gamma\ge2+ \beta$. For simplicity, we shall not keep track of all constants and simply write that $\mathsf a\lesssim\mathsf b$ and $\mathsf a\gtrsim\mathsf b$ if there is a positive constant $\mathsf c$ such that, respectively, $\mathsf a\le\mathsf b\,\mathsf c$ and $\mathsf a\ge\mathsf b\,\mathsf c$. We also define $\omega_d:=|\mathbb S^{d-1}|$ where $\mathbb S^{d-1}$ denotes the unit sphere in $\R^d$.

%%%%%%%%%%%%%%%%%%%%%%%%%%%%%%%%%%%%%%%%%%%%%%%%%%%%%%%%%%%%%%%%%%%%%%
\subsection{Outline of the method and key intermediate estimates}

Assume that $f$ is a finite mass solution of~\eqref{eq:main} on $\R_+\times\R^d\times\R^d$. Our goal is to relate
\[
\mathsf H[f]:=\int_{\R^d}\mathsf H_\xi[\Ff]\,\d\xi
\]
and
\[
-\,\dt\mathsf H[f]=-\,2\,\irdxmu{f\,\mathsf L f}+\delta\rint\mathsf R_\xi[\Ff]\,\d\xi
\]
by a differential inequality and use a Gr\"onwall estimate. According to Lemma~\ref{Lem:Equivalence}, the decay rate of $\n f^2$ is the same as for $\mathsf H_\xi[\Ff]$. Under Assumption~{\rm (H)}, we consider a solution $f$ of~\eqref{eq:main} with initial condition $f^{\mathrm{in}}\in\mathrm L^1(\d x\,\d v)\cap\mathrm L^2(\d x\,\d\mu)$. The main steps of our method are as follows:

\medskip\noindent$\rhd$ \emph{The solution is bounded in a weighted $\mathrm L^2$ space}. We shall prove the following result in Section~\ref{Sec:WeightedSpaces}.
%---------------------------------------------------------------------
\begin{proposition}\label{prop:propag_L2_m} Assume that~{\rm (H)} holds. Let $d\ge1$, $\gamma>0$, $\gamma + \beta \ge 0$, $k\in(0,\gamma)$ and $f$ be a solution of~\eqref{eq:main} with initial condition $f^{\mathrm{in}}\in\mathrm L^2(\bangle v^k\d x\,\d\mu)$. Then, there exists a positive constant $\mathcal C_k$ depending on $d$, $\gamma$, $\beta$ and $k$ such that
\[
\forall\,t\ge0\,,\quad\nrm{f(t,\cdot,\cdot)}{\mathrm L^2(\bangle v^k\d x\,\d\mu)}\le\mathcal C_k\,\nrm{f^{\mathrm{in}}}{\mathrm L^2(\bangle v^k\d x\,\d\mu)}\,.
\]
\end{proposition}
%---------------------------------------------------------------------

\medskip\noindent$\rhd$ \emph{The collision term controls the distance to the local equilibrium}. We have the following microscopic coercivity estimate.
%---------------------------------------------------------------------
\begin{proposition}\label{prop:WPII} Let $d\ge1$, $\gamma>0$, $\gamma + \beta \ge 0$, $\eta\in[- \beta,\gamma)$ and $k\in(0,\gamma)$. Assume that $ \beta=\,2$ if $\mathsf L=\mathsf L_1$, that Assumptions~\eqref{hyp:b_mass}--\eqref{hyp:b_bounds3} hold if $\mathsf L=\mathsf L_2$, and that $\sigma\in(0,2)$, $\beta=\sigma - \gamma$ if $\mathsf L=\mathsf L_3$. Then there exists a positive constant $\mathcal C$ depending on $\nrm f{\mathrm L^2(\d x\,\d\mu)}$ such that for any $f\in\mathrm L^2(\bangle v^k\d x\,\d\mu)$,
\[
\mathcal C\,\nrm{(1-\mathsf\Pi_\eta)f}{\mathrm L^2(\d x\,\bangle v^\eta\d\mu)}^{2\,\frac{k + \beta}{k-\eta}}\,\nrm f{\mathrm L^2(\d x\,\bangle v^k\d\mu)}^{-\,2\,\frac{\eta+\beta}{k-\eta}}\le-\irdxmu{f\,\mathsf L f}\,.
\]
\end{proposition}
%---------------------------------------------------------------------
This estimate is the extension of~\eqref{GeneralInterpolation} to the non-homogeneous case. The proof is done in Section~\ref{sec:WPII}. We shall use Proposition~\ref{prop:WPII} with $\eta=- \beta$ if $\gamma> \beta$ and for some $\eta\in(-\,\gamma,0)$ if $\gamma\le \beta$. The case $\eta\ge0$ is needed only in Step 4 of the proof of Proposition~\ref{prop:entropy_decay_A}.

\medskip\noindent$\rhd$ Our proofs require the computation of a large number of coefficients and various estimates which are collected in Sections~\ref{Sec:mu-lambda} and~\ref{Sec::equiv_symbol_mu_and_lambda}.

\medskip\noindent$\rhd$ \emph{A microscopic coercivity estimate} is established in Section~\ref{Sec:macro}, which goes as follows. Let us define the function
\[
\mathcal L(\xi):=\frac{|\xi|^{\alpha }}{\bangle\xi^{\alpha }}\quad\mbox{if}\quad \gamma\neq2 +\beta\,,\quad\mathcal L(\xi):=\frac{|\xi|^2\,\big|\log|\xi|\big|}{1+|\xi|^2\,\log|\xi|}\quad\mbox{if}\quad \gamma=2+\beta\,.
\]
%---------------------------------------------------------------------
\begin{proposition}\label{prop:entropy_decay_A} Let $\gamma>\max\{0,- \beta\}$ and $\eta\in(-\,\gamma,\gamma)$ such that $\eta \ge- \beta$. Under Assumption~{\rm (H)}, there exists a positive, bounded function $\xi\mapsto\mathcal K(\xi)$ such that
\[
\mathsf R_\xi[\Ff]\gtrsim\mathcal L(\xi)\,\n{\mathsf\Pi\Ff}^2-\mathcal K(\xi)\,\nrm{(1-\mathsf\Pi)\Ff}\eta^2\,.
\]
\end{proposition}
%---------------------------------------------------------------------
In Section~\ref{Sec:fractionalNash}, inspired by \emph{fractional Nash inequalities}, we deduce from~Proposition~\ref{prop:entropy_decay_A} an estimate on the distance in the direction which is orthogonal to the local equilibria.
%---------------------------------------------------------------------
\begin{corollary}\label{Cor:fractionalNash2} Under Assumption~{\rm (H)}, we have
\[
\rint\mathsf R_\xi[\Ff]\,\d\xi\;\gtrsim\;\nrm{\mathsf\Pi f}{\mathrm L^2(\d x\,\d\mu)}^{2\,(1+\frac{\alpha }d)}-\vertiii{(1-\mathsf\Pi)f}_{- \beta}^2\quad\mbox{if}\quad \gamma\neq 2 + \beta\,,
\]
\begin{multline*}
\rint\mathsf R_\xi[\Ff]\,\d\xi\;\gtrsim\;\nrm{\mathsf\Pi f}{\mathrm L^2(\d x\,\d\mu)}^{2\,(1+\frac{\alpha }d)}\,\log\(\frac{\;\nrm{\mathsf\Pi f}{\mathrm L^2(\d x\,\d\mu)}}{\;\nrm f{\mathrm L^1(\d x\,\d\mu)}}\)-\vertiii{(1-\mathsf\Pi)f}_{- \beta}^2\\\mbox{if}\quad \gamma=2+ \beta\,.
\end{multline*}
\end{corollary}
%---------------------------------------------------------------------
The proof is a straightforward consequence of Lemma~\ref{lem:fractionalNash2} if $\gamma \neq 2+ \beta$ and of Lemma~\ref{lem:fractionalNash3} if $\gamma=2+ \beta$. See details in Section~\ref{Sec:fractionalNash} and~\ref{Sec:fractionalNashLimit}.

%%%%%%%%%%%%%%%%%%%%%%%%%%%%%%%%%%%%%%%%%%%%%%%%%%%%%%%%%%%%%%%%%%%%%%
\subsection{Sketch of the proof of the main results}\label{Sec:Proofs}

The difficult part of the paper is the proof of Propositions~\ref{prop:propag_L2_m},~\ref{prop:WPII} and~\ref{prop:entropy_decay_A}, and Corollary~\ref{Cor:fractionalNash2}. If $\gamma\le \beta$, we have to take $\eta\neq- \beta$ and use additional interpolation estimates: see Section~\ref{Sec:MoreProofs}. Otherwise, the proof of Theorems~\ref{th:main2} and~\ref{th:main1} is not difficult if $\gamma>\beta$ and can be done as follows.

Under Assumption~{\rm (H)}, a solution of~\eqref{eq:main} is such that
\[
-\,\dt\mathsf H[f]=-2\irdxmu{f\,\mathsf L f}-\delta\rint\mathsf R_\xi[\Ff]\,\d\xi\,.
\]
Let us assume that $\gamma\neq2+ \beta$ and $\gamma>\beta$. We rely on Proposition~\ref{prop:WPII}.
\\[4pt]
$\bullet$ If $\beta\le0$, with $\eta=- \beta$, we find that
\[
\rint\mathsf R_\xi[\Ff]\,\d\xi\gtrsim\nrm{\mathsf\Pi f}{\mathrm L^2(\d x\,\d\mu)}^{2\,(1+\frac{\alpha }d)}
-\vertiii{(1-\mathsf\Pi)f}_{- \beta}^2\,.
\]
We obtain
\begin{align*}
-\,\dt\mathsf H[f]&\gtrsim(1-\delta)\,\vertiii{(1-\mathsf\Pi)f}_{- \beta}^2+\delta\,\nrm{\mathsf\Pi f}{\mathrm L^2(\d x\,\d\mu)}^{2\,(1+\frac{\alpha }d)}\\
&\gtrsim(1-\delta)\,\nrm{(1-\mathsf\Pi)f}{\mathrm L^2(\d x\,\d\mu)}^2+\delta\,\nrm{\mathsf\Pi f}{\mathrm L^2(\d x\,\d\mu)}^{2\,(1+\frac{\alpha }d)}\\
&\gtrsim\mathsf H[f]^{2\,(1+\frac{\alpha }d)}
\end{align*}
using the simple observation that $\nrm{(1-\mathsf\Pi)f}{- \beta}^2\ge\nrm{(1-\mathsf\Pi)f}{\mathrm L^2(\d\mu)}^2$ if $\beta\le0$.
\\[4pt]
$\bullet$ If $\beta\in(0,\gamma)$ and $2+\beta \neq \gamma > \beta$, again with $\eta=- \beta$, we find that
\[
-\,\dt\mathsf H[f]\gtrsim(1-\delta)\vertiii{(1-\mathsf\Pi)f}_{- \beta}^2+\delta\,\nrm{\mathsf\Pi f}{\mathrm L^2(\d x\,\d\mu)}^{2\,(1+\frac{\alpha }d)}\,.
\]
Using H\"older's inequality
\[
\n{(1-\mathsf\Pi)f}^2\le\nrm{(1-\mathsf\Pi)f}{- \beta}^\frac{2\,k}{k+\beta}\,\nrm{(1-\mathsf\Pi)f}k^\frac{\beta - 2}{k+\beta}\,,
\]
we conclude that
\[
-\,\dt\mathsf H[f]\gtrsim(1-\delta)\nrm{(1-\mathsf\Pi)f}{\mathrm L^2(\d x\,\d\mu)}^{2\,(1+\frac{\beta}k)}+\delta\,\nrm{\mathsf\Pi f}{\mathrm L^2(\d x\,\d\mu)}^{2\,(1+\frac{\alpha }d)}\,.
\]
$\bullet$ If $d\ge1$, $\beta\ge0$ and $\gamma=2+\beta$, $\alpha =2$ but there is a logarithmic correction in the expression of $\rint\mathsf R_\xi[\Ff]\,\d\xi$, which is responsible for the $O(\log t)$ correction of Theorems~\ref{th:main2} and~\ref{th:main1} in case $\gamma=2+\beta$ as $t\to+\infty$.
\\[4pt]
$\bullet$ For integrability reasons, the case $\gamma\le  \beta$ requires further estimates involving some $\eta\in(-\,\gamma,0)$ that will be dealt with in Sections~\ref{Sec:Interp_ext} and~\ref{Sec:beta+gamma<0}. Except in this case, the proof of Theorems~\ref{th:main2} and~\ref{th:main1} is complete.

%%%%%%%%%%%%%%%%%%%%%%%%%%%%%%%%%%%%%%%%%%%%%%%%%%%%%%%%%%%%%%%%%%%%%%
%%%%%%%%%%%%%%%%%%%%%%%%%%%%%%%%%%%%%%%%%%%%%%%%%%%%%%%%%%%%%%%%%%%%%%
\section{Estimates in weighted \texorpdfstring{$\mathrm L^2$}{L2} spaces}\label{Sec:WeightedSpaces}

In this section, we assume that $\beta\ge0$.

%%%%%%%%%%%%%%%%%%%%%%%%%%%%%%%%%%%%%%%%%%%%%%%%%%%%%%%%%%%%%%%%%%%%%%
\subsection{A result in weighted \texorpdfstring{$\mathrm L^2$}{L2} spaces}

Let us prove Proposition~\ref{prop:propag_L2_m}, \emph{i.e.}, the propagation of weighted norms $\mathrm L^2(\bangle v^k\d x\,\d\mu)$ with power law of order $k\in(0,\gamma)$.

The conservation of weighted norms has also been used in~\cite{BDLS} when $\M$ has a sub-exponential form. In that case, any value of~$k$ was authorized, and this was implicitly a consequence of the fact that such a local equilibrium~$\M$ had finite weighted norms $\mathrm L^2(\bangle v^k\d x\,\d\mu)$ for any $k\in\R_+$. For a local equilibrium given by~\eqref{eq:hypfattailed}, there is a limitation on~$k$ as we cannot expect a global propagation of higher moments than those of~$\M$.

For any function $h\in\mathrm L^2(\bangle v^k\d x\,\d\mu)$, one can notice that
\[
\nrm h{\mathrm L^2(\bangle v^k\d x\,\d\mu)}=\nrm{\M^{-1} h}{\mathrm L^2(\M\bangle v^k\d x\,\d v)}\,.
\]
In other words, it is equivalent to control the semi-group $e^{(\mathsf L-\mathsf T)t}$ in $\mathrm L^2(\bangle v^k\d x\,\d\mu)$ and $\M^{-1}\,e^{(\mathsf L-\mathsf T)t}$ in $\mathrm L^2(\M\bangle v^k\d x\,\d v)$. Since $\mathrm L^2(\bangle v^k\M\,\d x\,\d v)$ is a space interpolating between $\mathrm L^1(\M\bangle v^k\d x\,\d v)$ and $\mathrm L^\infty(\d x\,\d v)$ (see~\cite[Theorem~(2.9)]{stein_interpolation_1958}), we shall establish the result of Proposition~\ref{prop:propag_L2_m} by proving that $\M^{-1}\,e^{(\mathsf L-\mathsf T)t}$ is bounded onto $\mathrm L^\infty(\d x\,\d v)$ in Section~\ref{Sec:Linfinity} and onto $\mathrm L^1(\M\bangle v^k\d x\,\d v)$ in Section~\ref{Sec:L1}. In order to prove this last estimate, as in~\cite{kavian_fokker-planck_2015,lafleche_fractional_2020,BDLS}, we shall use a Lyapunov function method in Section~\ref{Sec:LyapunovFunction} and a splitting of the operator in Section~\ref{Sec:Splitting}.

%%%%%%%%%%%%%%%%%%%%%%%%%%%%%%%%%%%%%%%%%%%%%%%%%%%%%%%%%%%%%%%%%%%%%%
\subsection{The boundedness in \texorpdfstring{$\mathrm L^\infty(\d x\,\d v)$}{L infinity(dxdv)}}\label{Sec:Linfinity}

%---------------------------------------------------------------------
\begin{lemma} Let $d\ge1$ and $\gamma>0$. If~{\rm (H)} holds, then
\[
\forall\,t\ge0\,,\quad\nrm{\M^{-1}e^{t(\mathsf L-\mathsf T)}}{\mathrm L^\infty(\d x\,\d v)\to\mathrm L^\infty(\d x\,\d v)}\le1\,,
\]
where the norm $\nrm{\cdot}{X\to Y}$ denotes the operator norm for an operator with domain~$X$ and codomain $Y$.
\end{lemma}
%---------------------------------------------------------------------
\begin{proof} This is a consequence of the maximum principle in Case~$\mathsf L=\mathsf L_1$. In Case~$\mathsf L=\mathsf L_2$, $h^\#(t,x,v)=\M^{-1}(v)\,f(t,x+v\,t,v)$ solves
\[
\partial_th^\#+\nu(v)\,h^\#=\rint\mathrm b(v,v')\,\M(v')\,h^\#(t,x,v')\,\d v'\,,
\]
which is clearly a positivity preserving equation. The positivity of
\[
(t,x,v)\mapsto\nrm{h(0,\cdot,\cdot)}{\mathrm L^\infty(\d x\,\d v)}-h^\#(t,x,v)
\]
is also preserved, as it solves the same equation, which proves the claim. Case~$\mathsf L=\mathsf L_3$ is less standard as it relies on the maximum principle for fractional operators. As this is out of the scope of the present paper, we will only sketch the main steps of a proof. First of all, the results of~\cite{lafleche_fractional_2020} can be adapted to $E$ as defined by~\eqref{eq:def_E}, thus proving that the evolution according to $\partial_t-\M\,\mathsf L_3(\M^{-1}\cdot)$ preserves $\mathrm L^\infty$ bounds. This is also the case of $\partial_t-\mathsf T$. We can then conclude using a time-splitting approximation scheme of evolution and a Trotter formula.\Qed\end{proof}

%%%%%%%%%%%%%%%%%%%%%%%%%%%%%%%%%%%%%%%%%%%%%%%%%%%%%%%%%%%%%%%%%%%%%%
\subsection{A Lyapunov function method}\label{Sec:LyapunovFunction}

The boundedness of the operator $\M^{-1}e^{t(\mathsf L-\mathsf T)}$ in $\mathrm L^1(\M\bangle v^k\d x\,\d v)$ is equivalent to the boundedness of the operator $e^{t(\mathsf L-\mathsf T)}$ in $\mathrm L^1(\bangle v^k\d x\,\d v)$. To obtain such a bound, we rely on a Lyapunov function estimate.
%---------------------------------------------------------------------
\begin{lemma}\label{lem:Lyapunov} Let $d\ge1$, $\beta \geq 0$ and $\gamma > 0$. If~{\rm (H)} holds, then for any $k\in[0,\gamma + \beta)$, there exists $(a,b,R)\in\R\times\R_+\times\R_+$ such that for any $f\in\mathrm L^1(\bangle v^k\d x\,\d v)$,
\[\label{eq:Lyapunov}
\riint\tfrac f{|f|}\,\mathsf L f\,\bangle v^k\d x\,\d v\le\riint\(a\,\mathds1_{B_R}-b\bangle v^{- \beta}\)|f| \bangle v^k\d x\,\d v\,.
\]
As a special case corresponding to $k=0$, we have $\riint\tfrac f{|f|}\,\mathsf L f\,\d x\,\d v\le0$.
\end{lemma}
%---------------------------------------------------------------------
Here by convention, we shall write that $\tfrac f{|f|}=0$ if $f=0$.
\begin{proof} First assume that $f\ge 0$. Then one may write,
\begin{multline*}
\riint\mathsf Lf\,\bangle v^k\d x\,\d v=\riint\mathsf Lf\,\M \bangle v^k\d x\,\d\mu\\
=\riint\mathsf L^*(\M \bangle v^k)\,f\,\d x\,\d\mu\,.
\end{multline*}

\noindent$\bullet$ In Case~$\mathsf L=\mathsf L_1$, we notice that $\mathsf L$ is self-adjoint on $\mathrm L^2(\d\mu)$, recall that $ \beta= 2$ and compute
\begin{align*}
\M^{-1}\,\mathsf L_1\big(\M\bangle{\cdot}^k\big)(v)&=\bangle v^{d+\gamma} \nabla_v\cdot\(\bangle v^{-d-\gamma}\nabla_v\bangle v^k\)\\
&=k\bangle v^{d+\gamma}\,\nabla_v\cdot\(\bangle v^{-d-\gamma+k-2} v\)\\
&=k\(d+\gamma-k+2\)\bangle v^{k-4}-k\(\gamma+2-k\)\bangle v^{k-2}
\end{align*}
and obtain the result for any $k<\gamma + \beta=\gamma+2$.\medskip

\noindent$\bullet$ In Case~$\mathsf L=\mathsf L_2$, by Assumption~\eqref{hyp:b_mass} one obtains that
\begin{align*}
\M^{-1}\,\mathsf L_2^*\big(\M \bangle{\cdot}^k\big)(v)&=\rint\mathrm b(v',v)\(\bangle{v'}^k\M(v') -\bangle v^k \M(v')\)\d v'\\
&=\(\rint\mathrm b(v',v)\frac{\bangle{v'}^k}{\bangle v^k}\,\M(v')\,\d v'-\nu(v)\)\bangle v^k\,.
\end{align*}
By Assumption~\eqref{hyp:b_bounds2}, $\mathcal C_{\mathrm b}(k)=\sup_{v\in\R^d}\bangle v^{\beta}\rint\mathrm b(v',v)\bangle{v'}^k \M(v')\,\d v'$ is finite for any $k\in(0,\gamma + \beta)$, and as a consequence, we know that
\[
\quad\forall\,v\in\R^d\,,\quad\nu(v)\le\rint\mathrm b(v',v)\bangle{v'}^k \M(v')\,\d v'\le\mathcal C_{\mathrm b}(k) \bangle v^{- \beta} .
\]
This yields
\[
\M^{-1}\,\mathsf L_2^*\big(\M\,\bangle{\cdot}^k\big)(v)\le\(\frac{\mathcal C_{\mathrm b}(b)}{\bangle v^k}-\frac{\nu(v)}{\bangle v^{- \beta}}\)\bangle v^{- \beta}\,.
\]
We conclude that Inequality~\eqref{eq:Lyapunov} holds for any $k\in(0,\gamma +\beta)$ by Assumption~\eqref{hyp:b_beta}.\medskip

\noindent$\bullet$ In Case~$\mathsf L=\mathsf L_3$, it is elementary to compute $\mathsf L_3^*$ and observe that
\begin{align*}
\M^{-1}\,\mathsf L_3^*\big(\M\,\bangle{\cdot}^k\big)(v)&=\Delta_v^{\sigma/2}\bangle v^k-E(v)\,\cdot\nabla_v\bangle v^k\\
&=\left[\bangle v^{-k}\Delta_v^{\sigma/2}\bangle v^k-k\(v\cdot E\)\bangle v^{-2}\right] \bangle v^k,\\
&\le\left[\bangle v^{-k}\Delta_v^{\sigma/2}\bangle v^k-C\,\bangle v^{- \beta}\right] \bangle v^k,
\end{align*}
where the estimate $k\(v\cdot E\)\bangle v^{-2}\ge C\,\bangle v^{- \beta}$ for some $C>0$ arises as a consequence of Proposition~\ref{prop:estE}. According to~\cite[Lemma~3.1]{biler_blowup_2009} (also see~\cite{biler_blowup_2010,lafleche_fractional_2020}), we~have
\[
\forall\,v\in\R^d\,,\quad\Delta_v^{\sigma/2}\bangle v^k\lesssim \bangle v^{k-\sigma}\,,
\]
under the condition that $k<\sigma=\gamma +\beta$. This again completes the proof of Inequality~\eqref{eq:Lyapunov}.

\medskip When $f$ changes sign, it is possible to reduce the problem to the case $f\ge 0$ as follows. In Case~$\mathsf L=\mathsf L_1$, we use Kato's inequality to assert that
\[
\frac{f}{|f|}\,\Delta_vf\le\Delta_v|f|
\]
in the sense of Radon measures (see~\cite[Lemma~A]{MR333833} or, for instance,~\cite[Theorem~1.1]{brezis_katos_2004}). Case~$\mathsf L=\mathsf L_2$ relies on the elementary observation that
\begin{align*}
\riint\tfrac f{|f|}\,\mathsf L_2f \bangle v^k\d v\,\d v'&=\riint\mathrm b(v,v')\,f'\,\frac f{|f|}\,\M\bangle v^k\d v\,\d v'-\rint\nu\,|f|\,\d v\\
&\le\riint\mathrm b(v,v')\,|f'|\,\M\bangle v^k\d v\,\d v'-\rint\nu\,|f|\,\d v\,.
\end{align*}
In Case~$\mathsf L=\mathsf L_3$, the result follows from Kato's inequality extended to the fractional Laplacian as follows. Let us consider $\varphi_\varepsilon(s)=\sqrt{\varepsilon^2+s^2}$ and notice that
\begin{multline*}
\(\Delta_v^{\sigma/2}\varphi_\varepsilon(f)\)(v)-\varphi_\varepsilon'(f(v))\(\Delta_v^{\sigma/2}f\)(v)\\
=C_{d,\sigma}\iint_{\R^d}\frac{\varphi_\varepsilon(f(v'))-\varphi_\varepsilon(f(v))-\varphi_\varepsilon'(f(v))\,(f(v')-f(v))}{|v'-v|^{d+\sigma}}\,\d v\ge0
\end{multline*}
because $\varphi_\varepsilon$ is convex since $\varphi_\varepsilon''(s)=\varepsilon^2\,(\varepsilon^2+s^2)^{-3/2}$ and according for example to \cite[Chapter~2]{landkof_foundations_1972}
\be{Garofalo}
C_{d,\sigma}=-\,\frac{2^{\sigma}}{\pi^{d/2}}\,\frac{\Gamma\big(\frac{d+\sigma}2\big)}{\Gamma\big(-\frac\sigma2\big)}> 0\,.
\ee
By passing to the limit as $\varepsilon\to0$, we obtain
\[
\frac{f}{|f|}\,\Delta_v^{\sigma/2}f\le \Delta_v^{\sigma/2}|f|\,.
\]
In all cases, with $\mathsf L=\mathsf L_i$, $i=1$, $2$, $3$, we have
\[
\rint\frac f{|f|}\,\mathsf Lf \bangle v^k\d x\,\d v\le\rint\(\mathsf L|f|\)\bangle v^k\d x\,\d v
\]
and the problem is reduced to the case of a nonnegative distribution function~$f$.\Qed\end{proof}

%%%%%%%%%%%%%%%%%%%%%%%%%%%%%%%%%%%%%%%%%%%%%%%%%%%%%%%%%%%%%%%%%%%%%%
\subsection{A splitting of the evolution operator}\label{Sec:Splitting}

We rely on the strategy of~\cite{gualdani_factorization_2017,kavian_fokker-planck_2015,mischler_exponential_2016} by writing $\mathsf L-\mathsf T$ as the sum of a dissipative part $\mathsf C$ and a bounded part $\mathsf B$ such that $\mathsf L-\mathsf T=\mathsf B+\mathsf C$.
%---------------------------------------------------------------------
\begin{lemma}\label{lem:splitting}
Under the assumptions of Lemma~\ref{lem:Lyapunov}, let $(k,k_*)\in(0,\gamma)\times(0,\gamma+\beta)$ be such that $k_*>k+\beta$, $a=\max\{a_k,a_{k_*}\}$, $R=\min\big\{R_k,R_{k_*}\big\}$, $\mathsf C:=a\,\mathds1_{B_R}$ and $\mathsf B:=\mathsf L-\mathsf T-\mathsf C$. Then for any $t\in\R_+$, we have:
\begin{enumerate}
\item[{\rm (i)}] $\|\mathsf C\|_{\mathrm L^1(\d x\,\d\mu)\to\mathrm L^1\(\bangle v^{k_*}\,\d x\,\d\mu\)}\le a\,(1+R^2)^{k_*/2}$,
\item[{\rm (ii)}] $\|e^{t\mathsf B}\|_{\mathrm L^1\(\bangle v^k\d x\,\d\mu\)\to\mathrm L^1\(\bangle v^k\d x\,\d\mu\)}\le1$,
\item[\rm{(iii)}] $\|e^{t\mathsf B}\|_{\mathrm L^1\(\bangle v^{k_*}\,\d x\,\d\mu\)\to\mathrm L^1\(\bangle v^k\d x\,\d\mu\)}\le c\(1+t\)^{\frac{k-k_*}{\beta}}$ for some $c>0$.
\end{enumerate}
\end{lemma}
%---------------------------------------------------------------------
\begin{proof} Property (i) is a consequence of the definition of $\mathsf C$. Property (ii) follows from Lemma~\ref{lem:Lyapunov}. Indeed, for any $g\in\mathrm L^1(\bangle v^k\d x\,\d v)$,
\begin{align*}
\riint\kern-2pt\frac g{|g|}\,\mathsf B\,g\,\bangle v^k\d x\,\d v&\le\riint\kern-2pt\(a_k\,\mathds1_{B_{R_k}}\kern-4.5pt-a\,\mathds1_{B_R}\!-b_k\bangle v^{- \beta}\)|g| \bangle v^k\d x\,\d v\\
&\le-\,b_k\,\nrm g{\mathrm L^1\(\bangle v^{k- \beta} \d x\,\d v\)}\,.
\end{align*}
To prove (iii), define $g:=e^{t\mathsf B}\,g^{\mathrm{in}}$. By H\"older's inequality, we get
\[
\nrm g{\mathrm L^1\(\bangle v^k\d v\,\d x\)}\le\nrm g{\mathrm L^1\(\bangle v^{k- \beta}\,\d x\,\d v\)}^{\frac{k_*-k}{k_*-k +\beta}}\,\nrm{g^{\mathrm{in}}}{\mathrm L^1\(\bangle v^{k_*}\,\d x\,\d v\)}^\frac{\beta}{k_*-k + \beta}
\]
and, as a consequence of the above contraction property,
\[
\riint\frac g{|g|}\,\mathsf B\,g \bangle v^k\d x\,\d v\le-\,b_k\,\(\nrm g{\mathrm L^1\(\bangle v^k\d v\,\d x\)}\)^{1+\frac{ \beta }{k_*-k}}\,\nrm{g^{\mathrm{in}}}{\mathrm L^1\(\bangle v^{k_*}\,\d x\,\d v\)}^{-\,\frac{ \beta }{k_*-k}}\,,
\]
so that by Gr\"onwall's lemma, we obtain
\begin{align*}
\nrm g{\mathrm L^1(\bangle v^k\d x\,\d v)}&\le\(\nrm{g^{\mathrm{in}}}{\mathrm L^1(\bangle v^k\d x\,\d v)}^{-\,\frac{\beta}{k_*-k}}+\tfrac{b_k\,\beta}{k_*-k}\,\nrm{g^{\mathrm{in}}}{\mathrm L^1(\bangle v^{k_*} \d x\,\d v)}^{-\,\frac{\beta}{k_*-k}}\,t\)^{\!-\,\frac{k_*-k}{\beta}}\\
&\le\(1+\tfrac{k_*-k}{b_k\,\beta}\,t\)^{-\frac{k_*-k}{\beta}}\,\nrm{g^{\mathrm{in}}}{\mathrm L^1(\bangle v^{k_*} \d x\,\d v)}\,.
\end{align*}
\Qed\end{proof}

%%%%%%%%%%%%%%%%%%%%%%%%%%%%%%%%%%%%%%%%%%%%%%%%%%%%%%%%%%%%%%%%%%%%%%
\subsection{The boundedness in \texorpdfstring{$\mathrm L^1(\M\bangle v^k\d x\,\d v)$}{L1(F(v)<v>kdxdv)}}\label{Sec:L1}

%---------------------------------------------------------------------
\begin{lemma}\label{lem:propag_L1_m} Let $d\ge1$, $\gamma>0$ and $\beta \ge 0$, $k\in(0,\gamma)$ and assume that~{\rm (H)} holds. There exists a positive constant $\mathcal C_k$ such that, for any solution $f$ of~\eqref{eq:main} with initial condition $f^{\mathrm{in}}\in\mathrm L^1(\bangle v^k\d x\,\d v)$,
\[
\forall\,t\ge0\,,\quad\nrm{f(t,\cdot,\cdot)}{\mathrm L^1(\bangle v^k\d x\,\d v)}\le\mathcal C_k\,\nrm{f^{\mathrm{in}}}{\mathrm L^1(\bangle v^k\d x\,\d v)}\,.
\]
\end{lemma}
%---------------------------------------------------------------------
\begin{proof} Let us consider the Duhamel formula
\[
e^{t(\mathsf L-\mathsf T)}=e^{t\mathsf B}+\int_0^t\mathsf e^{(t-s)\mathsf B}\,\mathsf C\,e^{s(\mathsf L-\mathsf T)}\,\d s\,.
\]
By Lemma~\ref{lem:Lyapunov}, we know that
\[
\|e^{t(\mathsf L-\mathsf T)}\|_{\mathrm L^1(\bangle v^k\d x\,\d v)\to\mathrm L^1(\bangle v^k\d x\,\d v)}\le a\,(1+R^2)^{k_*/2}\,.
\]
Using the estimates of Lemma~\ref{lem:splitting}, we get
\[
\|e^{t(\mathsf L-\mathsf T)}\|_{\mathrm L^1(\bangle v^k\d x\,\d v)\to\mathrm L^1(\bangle v^k\d x\,\d v)}\le1+a\,c\,(1+R^2)^{k_*/2}\int_0^t\(1+s\)^{-\frac{k_*-k}{\beta}}\d s\,,
\]
which is bounded uniformly in time with the choice $k_*-k> \beta$.
\Qed\end{proof}

%%%%%%%%%%%%%%%%%%%%%%%%%%%%%%%%%%%%%%%%%%%%%%%%%%%%%%%%%%%%%%%%%%%%%%
%%%%%%%%%%%%%%%%%%%%%%%%%%%%%%%%%%%%%%%%%%%%%%%%%%%%%%%%%%%%%%%%%%%%%%
\section{Interpolation inequalities}\label{sec:WPII}

We refer to~\cite{wang_simple_2014} for a general strategy for proving~\eqref{GeneralInterpolation} which applies in particular to $\mathsf L_3$ in the case~$\mathsf L=\mathsf L_3$. However, for the operators considered in this paper, direct estimates can be obtained as follows.

%%%%%%%%%%%%%%%%%%%%%%%%%%%%%%%%%%%%%%%%%%%%%%%%%%%%%%%%%%%%%%%%%%%%%%
\subsection{Hardy-Poincar\'e inequality and consequences}

%---------------------------------------------------------------------
\begin{lemma}\label{Lem:Hardy-Poincare} Let $d\ge1$ and $\gamma>0$. We have the \emph{Hardy-Poincar\'e inequality}
\[
\forall\,h\in\mathrm L^2(\bangle v^{-2}\M\,\d v)\,,\quad\irdv{|\nabla_vh|^2\,\M}\ge2\,(d+\gamma)\irdv{|h-\overline{h}_{-2}|^2\bangle v^{-2}\M}
\]
with $\overline{h}_{-2}:=\frac{\irdv{h\bangle v^{-2}\M}}{\irdv{\bangle v^{-2}\M}}$. \end{lemma}
%---------------------------------------------------------------------
See~\cite{MR2481073} for a proof. We deduce the following interpolation inequality.
%---------------------------------------------------------------------
\begin{corollary}\label{Cor:Hardy-Poincare}
Let $d\ge1$, $\gamma>0$ and $k\in(0,\gamma)$. There exists a positive constant~$\mathcal C_1$ such that, for any $f\in\mathrm L^2(\bangle v^k\d x\,\d\mu)$ such that $\nabla_vh\in\mathrm L^2(\d x\,\d\mu)$ where $h=f/\M$, we have the inequality
\[
\mathcal C_1\,\nrm{(1-\mathsf\Pi)f}{\mathrm L^2(\d x\,\d\mu)}^{2+\frac4k}\le\(\irdxv{|\nabla_vh|^2\,\M}\)\,\nrm f{\mathrm L^2(\bangle v^k\d x\,\d\mu)}^\frac4k\,.
\]
\end{corollary}
%---------------------------------------------------------------------
\begin{proof}
Let $\overline{h}_0:=\irdv{h\,\M}$ and observe that
\[
\irdv{|h-\overline{h}_0|^2\,\M}=\inf_{H\in\R}\irdv{|h-H|^2\,\M} \leq\irdv{|h-\overline{h}_{-2}|^2\,\M}\,.
\]
Setting $g=h-\overline{h}_{-2}$, we deduce on the one hand from the Cauchy-Schwarz inequality that
\begin{multline*}
\irdxv{|h-\overline{h}_{-2}|^2\,\M}=\irdxv{|g|^2\,\M}\\
\le\(\irdxv{|g|^2\,\frac\M{\bangle v^2}}\)^\frac k{k+2}\(\irdxv{|g|^2\bangle v^k\M}\)^\frac2{k+2},
\end{multline*}
and we deduce that
\[
|g|^2\le\frac12\(|h|^2+\overline{h}_{-2}^2\)\le\frac12\(|h|^2+\frac{\overline{h}_0^2}{\irdv{\bangle v^{-2}\M}}\)
\]
using $\bangle v\ge1$ and the definition of $\overline{h}_{-2}$ on the other hand. Collecting these estimates with the result of Lemma~\ref{Lem:Hardy-Poincare} shows that
\[
\frac{2^{1+\frac2k}\,(d+\gamma)\,\nrm{(1-\mathsf\Pi)f}{\mathrm L^2(\d x\,\d\mu)}^{2+\frac4k}}{\(\nrm f{\mathrm L^2(\bangle v^k\d x\,\d\mu)}^2+\mathsf c_k\,\nrm{\mathsf\Pi f}{\mathrm L^2(\d x\,\d\mu)}^2\)^\frac2k}\le\irdv{|\nabla_vh|^2\,\M}
\]
where $\mathsf c_k:=\irdv{\bangle v^k\M}/\irdv{\bangle v^{-2}\M}$. This completes the proof after observing that
\[
\nrm{\mathsf\Pi f}{\mathrm L^2(\d x\,\d\mu)}\le\nrm f{\mathrm L^2(\d x\,\d\mu)}\le\nrm f{\mathrm L^2(\bangle v^k\d x\,\d\mu)}\,.
\]
\Qed\end{proof}

%%%%%%%%%%%%%%%%%%%%%%%%%%%%%%%%%%%%%%%%%%%%%%%%%%%%%%%%%%%%%%%%%%%%%%
\subsection{A gap inequality for the scattering operator}

Let $\mathsf L=\mathsf L_2$ be the scattering operator of Case~$\mathsf L=\mathsf L_2$.
%---------------------------------------------------------------------
\begin{lemma}\label{lem:ScatteringGap} Let $\gamma>\max\{0,- \beta\}$. Assume that~\eqref{hyp:b_mass} and~\eqref{hyp:b_bounds2} hold. Then we have
\[
\riint\mathrm b(v,v')\(h-h'\)^2\,\M\,\M'\,\d v\,\d v'\ge\Lambda\irdv{\left|h-\overline h_{- \beta}\right|^2 \bangle v^{- \beta}\,\M}
\]
for any $h\in\mathrm L^2(\d v)$, with $\Lambda:=\frac2Z\irdv{\M\bangle v^{- \beta}}$ and $\overline h_{- \beta}:=\frac{\irdv{h\,\M\bangle v^{- \beta}}}{\irdv{\M\bangle v^{- \beta}}}$.\end{lemma}
%---------------------------------------------------------------------
Notice that here we do not assume~\eqref{hyp:b_bounds2} or~\eqref{hyp:b_bounds3}, and consider any $\beta+\gamma >0$.
\begin{proof}
Using~\eqref{hyp:b_mass}, we have
\begin{multline*}
\hspace*{-8pt}2\riint\mathrm b(v,v')\(h'-h\)h\,\M\,\M'\,\d v\,\d v'\\
=\riint\mathrm b(v,v')\,\(2\,h'h-h^2\)\,\M\,\M'\,\d v\,\d v'-\riint\mathrm b(v',v)\, h^2 \M\,\M'\,\d v\,\d v'\,.
\end{multline*}
Exchanging variables $v$ and $v'$ gives
\begin{multline*}
-\riint\mathrm b(v,v')\(h'-h\)h\,\M\,\M'\,\d v\,\d v'\\
=\frac 12\riint\mathrm b(v,v')\(h-h'\)^2\,\M\,\M'\,\d v\,\d v'\,.
\end{multline*}
By assumption~\eqref{hyp:b_bounds2}, we know that $\mathrm b(v,v')\ge Z^{-1}\bangle v^{- \beta} \bangle{v'}^{- \beta}$ and observe that
\begin{multline*}
\riint\bangle v^{- \beta} \bangle{v'}^{- \beta}\(h-h'\)^2\,\M\,\M'\,\d v\,\d v'\\
\hspace*{-1cm}=2{\irdv{\M \bangle v^{- \beta}}}\irdmu{\frac{|f|^2}\M \bangle v^{- \beta}}-2\(\irdv{f \bangle v^{- \beta}}\)^2\\
=2\irdv{\M \bangle v^{- \beta}}\irdv{\left|\frac f\M-\frac{\irdv{f \bangle v^{- \beta}}}{\irdv{\M \bangle v^{- \beta}}}\right|^2 \bangle v^{- \beta}\,\M}\\
=\Lambda\,Z\irdv{\left|h-\overline h_{- \beta}\right|^2 \bangle v^{- \beta}\,\M}\,.
\end{multline*}
\Qed\end{proof}
Next we deduce the following interpolation inequality.
%---------------------------------------------------------------------
\begin{corollary}\label{Cor:ScatteringGap} Under the assumptions of Lemma~\ref{lem:ScatteringGap}, for any $k\in(0,\gamma)$, there exists a positive constant~$\mathcal C_2$ such that, for any $f\in\mathrm L^2(\bangle v^k\d x\,\d\mu)$,
\begin{align*}
\mathcal C_2\,\nrm{(1-\mathsf\Pi)f}{\mathrm L^2(\d x\,\d\mu)}^{2+2\,\frac{\beta}k}&\le\(-\irdxmu{f\,\mathsf L_2 f}\)\nrm f{\mathrm L^2(\bangle v^k\d x\,\d\mu)}^{2\,\frac{\beta}k}\quad\mbox{if}\quad \beta >0\,,\\
\mathcal C_2\,\nrm{(1-\mathsf\Pi)f}{\mathrm L^2(\d x\,\d\mu)}^2&\le-\irdxmu{f\,\mathsf L_2 f}\quad\mbox{if}\quad \beta\le 0\,.
\end{align*}
\end{corollary}
%---------------------------------------------------------------------
\begin{proof} With $h=f/\M$, we recall that, as a consequence of Lemma~\ref{lem:ScatteringGap},
\[
-\irdmu{f\,\mathsf L_2 f}=\riint\kern-4pt\mathrm b(v,v')\(h-h'\)^2\,\M\,\M'\,\d v\,\d v'\ge\frac\Lambda2\irdv{|g|^2 \bangle v^{- \beta}\,\M}\,,
\]
with $g:=h-\overline h_{- \beta}$. Moreover, we observe that
\be{eq:inf}
\irdv{|h-\overline h_0|^2\,\M}=\inf_{H\in\R}\irdv{|h-H|^2\,\M} \leq \irdv{|g|^2\,\M}.
\ee
Hence, if $\beta \leq 0$, the result follows from the fact that $\bangle v^{- \beta}\ge1$.

Assume next that $\beta > 0$. We deduce from \eqref{eq:inf} and the Cauchy-Schwarz inequality that
\begin{multline*}
\nrm{(1-\mathsf\Pi)f}{\mathrm L^2(\d x\,\d\mu)}^2 = \irdxv{|h-\overline h_0|^2\,\M} \leq \irdxv{|g|^2\,\M}\\
\le\(\irdxv{|g|^2 \bangle v^{- \beta}\,\M}\)^\frac k{k +\beta}\(\irdxv{|g|^2 \bangle v^k\M}\)^\frac{\beta}{k +\beta}.
\end{multline*}
and to control the last factor we notice that
\[
|g|^2\le\frac12\(|h|^2+\overline h_{- \beta}^2\)\le\frac12\(|h|^2+\frac{\overline h^2}{\irdv{|g|^2 \bangle v^{- \beta} \M}}\)
\]
using $\bangle v\ge1$ and the definition of $\overline h_{- \beta}$. Collecting these estimates with the result of Lemma~\ref{lem:ScatteringGap} completes the proof.
\Qed\end{proof}

%%%%%%%%%%%%%%%%%%%%%%%%%%%%%%%%%%%%%%%%%%%%%%%%%%%%%%%%%%%%%%%%%%%%%%
\subsection{Fractional Fokker-Planck operator: an interpolation inequality}

Let us compute $\irdmu{f\,\mathsf L_3 f}$. We recall that in Case~$\mathsf L=\mathsf L_3$, $\mathsf L_3 f=\Delta_v^{\sigma/2}f+\nabla_v\cdot\(E\,f\)$. With $h=f/\M$, we have
\[
\irdmu{f\,\Delta_v^{\sigma/2}f}=C_{d,\sigma}\irdvv{\frac{h^2-h\,h'}{|v-v'|^{d+\sigma}}\,\M}.
\]
On the other hand, we know that $h\,\nabla_v\cdot(f\,E)=h\,\nabla_v\cdot(h\,\M\,E)=\frac12\,\nabla_v (h^2)\cdot(\M\,E)+h^2\,\nabla_v\cdot(\M\,E)$ and after an integration by parts, we obtain
\[
\irdv{h\,\nabla_v\cdot(f\,E)}=\frac12\irdv{h^2\,\nabla_v\cdot(\M\,E)}=\frac12\irdv{h^2\,\Delta^{\sigma/2}\M}\,.
\]
After exchanging the variables $v$ and $v'$, we arrive at
\[
\irdv{h\,\nabla_v\cdot(f\,E)}=-\,\frac{C_{d,\sigma}}2\irdvv{\frac{h^2-(h')^2}{|v-v'|^{d+\sigma}}\,\M}.
\]
Altogether, this means that
\begin{multline*}
-\irdmu{f\,\mathsf L_3 f}=\frac{C_{d,\sigma}}2\!\irdvv{\frac{|h-h'|^2}{|v-v'|^{d+\sigma}}\,\M}\\
=\frac{C_{d,\sigma}}4\!\irdvv{\frac{|h-h'|^2}{|v-v'|^{d+\sigma}}\,(\M+\M')}\,.
\end{multline*}
%---------------------------------------------------------------------
\begin{corollary}\label{Cor:FFP}
Let $d\ge1$, $\gamma>0$, $\sigma\in(0,2)$, $\beta=\sigma - \gamma$ and $k\in(0,\gamma)$. With the notation of Corollary~\ref{Cor:ScatteringGap}, there exists a positive constant~$\mathcal C_3$ such that, for any $f\in\mathrm L^2(\bangle v^k\d x\,\d\mu)$, we have the inequality
\begin{align*}
\mathcal C_3\,\nrm{(1-\mathsf\Pi)f}{\mathrm L^2(\d x\,\d\mu)}^{2+2\,\frac{\beta}k}&\le\(-\irdxmu{f\,\mathsf L_3 f}\)\nrm f{\mathrm L^2(\bangle v^k\d x\,\d\mu)}^{2\,\frac{\beta}k}\quad\mbox{if}\quad \beta > 0\,,\\
\mathcal C_3\,\nrm{(1-\mathsf\Pi)f}{\mathrm L^2(\d x\,\d\mu)}^2&\le-\irdxmu{f\,\mathsf L_3 f}\quad\mbox{if}\quad \beta\le0\,.
\end{align*}
\end{corollary}
%---------------------------------------------------------------------
\begin{proof} From the elementary estimate
\[
\forall\,(v,v')\in\R^d\times\R^d\,,\quad\bangle v^{- \beta} \bangle{v'}^{- \beta} \M\,\M' \leq \kappa\,\frac{\M+\M'}{|v-v'|^{d+\sigma}}
\]
with $\kappa=c_\gamma\,\sup_{(v,v')\in\R^d\times\R^d}\frac{\bangle v^{d+\gamma}+\bangle{v'}^{d+\gamma}}{\bangle v^{d+\sigma} \bangle{v'}^{d+\sigma}}\,|v-v'|^{d+\sigma}$, we deduce that
\[
-\irdmu{f\,\mathsf L_2 f}\le -\,\kappa\irdmu{f\,\mathsf L_3 f}
\]
and conclude by Corollary~\ref{Cor:ScatteringGap} with $\mathcal C_3=\mathcal C_2/\kappa$.
\Qed\end{proof}

As a side result, let us observe that we obtain a \emph{fractional Poincar\'e inequality} as in~\cite[Corollary~1.2, (1)]{wang_functional_2015}, with an explicit constant, that goes as follows.
%---------------------------------------------------------------------
\begin{corollary}\label{Cor:FFPPoincare} Under the same assumptions as in Corollary~\ref{Cor:FFP},
\[
-\irdmu{f\,\mathsf L_3 f}\ge\kappa\,\Lambda\irdv{\left|h-\overline h_{- \beta}\right|^2 \bangle v^{- \beta} \M}
\]
where $\kappa$ is as in the proof of Corollary~\ref{Cor:FFP} and $\Lambda$ is the constant of Lemma~\ref{lem:ScatteringGap}. \end{corollary}
%---------------------------------------------------------------------

%%%%%%%%%%%%%%%%%%%%%%%%%%%%%%%%%%%%%%%%%%%%%%%%%%%%%%%%%%%%%%%%%%%%%%
\subsection{Convergence to the local equilibrium: microscopic coercivity}\label{Sec:Interp_ext}

We can summarize Lemma~\ref{Lem:Hardy-Poincare}, Lemma~\ref{lem:ScatteringGap} and Corollary~\ref{Cor:FFPPoincare} as
\[
\mathcal C\,\nrm{f-\overline h_{- \beta}\,\M}{- \beta}^2\le-\bangle{f\,,\mathsf L f}
\]
for positive constant $\mathcal C$, where $\overline h_{- \beta}=\irdv{h\,\M \bangle v^{- \beta}}/\irdv{\M \bangle v^{- \beta}}$ and $h=f/\M$. Here $\mathsf L=\mathsf L_1$, $\mathsf L_2$ or $\mathsf L_3$ respectively in Cases~$\mathsf L=\mathsf L_1$,~$\mathsf L=\mathsf L_2$ or~$\mathsf L=\mathsf L_3$. In the homogeneous case, an additional H\"older inequality establishes Inequality~\eqref{GeneralInterpolation} of Section~\ref{Sec:intro}. The same strategy can be applied in the non-homogeneous case after integrating with respect to $x\in\R^d$.

\begin{proof}[Proof of Proposition~\ref{prop:WPII}] It is a straightforward consequence of Proposition~\ref{prop:propag_L2_m} on the one hand, and of Corollaries~\ref{Cor:Hardy-Poincare},~\ref{Cor:ScatteringGap} and~\ref{Cor:FFP} if, respectively, $\mathsf L=\mathsf L_1$, $\mathsf L_2$ or $\mathsf L_3$.
\Qed\end{proof}

As an alternative formulation of Proposition~\ref{prop:WPII} and in preparation for the case $\gamma\le \beta$ (see Section~\eqref{Sec:beta+gamma<0}), let us collect some additional observations. The inequality
\[
\nrm{(1-\mathsf\Pi)f}\eta^2\le\(-\bangle{f,\mathsf Lf}\)^\theta\,\nrm{(1-\mathsf\Pi)f}k^{2\,(1-\theta)}
\]
with $\theta=\frac{k-\eta}{k+ \beta}$ can be rewritten with $\mathsf x_\zeta:=\nrm{(1-\mathsf\Pi)f}\zeta^2$ and $\mathsf z:=-\bangle{f,\mathsf Lf}$ as
\[
\mathsf x_\eta\le\mathsf z^\theta\,\mathsf x_k^{1-\theta}=\(R^{-1/\theta}\,\mathsf z\)^\theta\(R^{1/(1-\theta)}\mathsf x_k\)^{1-\theta}\le\theta\,R^{-1/\theta}\,\mathsf z+(1-\theta)\,R^{1/(1-\theta)}\mathsf x_k
\]
for any $R>0$, by Young's inequality. This amounts to
\[
\mathsf z\ge\frac1\theta\,R^\frac1\theta\,\mathsf x_\eta-\frac{1-\theta}\theta\,R^{\frac1\theta+\frac1{1-\theta}}\,\mathsf x_k=r\,\mathsf x_\eta-(1-\theta)\,\theta^\frac\theta{1-\theta}\,r^\frac1{1-\theta}\,\mathsf x_k\,.
\]
An integration with respect to $x$ shows the following result.
%---------------------------------------------------------------------
\begin{corollary}\label{Cor:WPIIext} Let $\theta=\frac{k-\eta}{k +\beta}$. Under the assumptions of Proposition~\ref{prop:WPII}, we have
\begin{multline}\label{Z}
-\irdx{\bangle{f,\mathsf Lf}}\ge r\,\nrm{(1-\mathsf\Pi)f}{\mathrm L^2(\d x\,\bangle v^\eta\d v)}^2\\
-(1-\theta)\,\theta^\frac\theta{1-\theta}\,r^\frac1{1-\theta}\,\nrm{(1-\mathsf\Pi)f}{\mathrm L^2(\d x\,\bangle v^k\d v)}^2
\end{multline}
for any $f\in\mathrm L^2(\bangle v^k\d x\,\d\mu)$ and for any $r>0$.\end{corollary}
%---------------------------------------------------------------------

%%%%%%%%%%%%%%%%%%%%%%%%%%%%%%%%%%%%%%%%%%%%%%%%%%%%%%%%%%%%%%%%%%%%%%
%%%%%%%%%%%%%%%%%%%%%%%%%%%%%%%%%%%%%%%%%%%%%%%%%%%%%%%%%%%%%%%%%%%%%%
\section{Hypocoercivity estimates}\label{Sec:hypoco}

%%%%%%%%%%%%%%%%%%%%%%%%%%%%%%%%%%%%%%%%%%%%%%%%%%%%%%%%%%%%%%%%%%%%%%
We start by defining some coefficients. With the notation $\nrm g\eta^2:=\!\irdmu{|g|^2 \bangle v^\eta}$, we define $\mu_\mathsf L$ and $\lambda_\mathsf L$ by
\[
\mu_\mathsf L(\xi):=\nrm{\mathsf L^*\big((v\cdot\xi)\,\varphi(\xi,v)\,\M\big)}{-\eta}^2\,\quad\lambda_\mathsf L:=\nrm{\mathsf L^*(\psi\,\M)}{-\eta}^2\,,
\]
for some parameter $\eta\in(-\,\gamma,\gamma)$, where $\mathsf L^*$ denotes the dual of $\mathsf L$ in $\mathrm L^2(\d\mu)$, and
\be{Notations:LambdaMu}
\begin{array}{l}
\lambda_k:=\irdv{|v\cdot\xi|^k\bangle v^{-2}\M}=\bangle{\M,\,|\mathsf T|^k\,\psi\,\M}\,,\\[4pt]
\mu_k:=\irdv{|v\cdot\xi|^k\,\varphi(v)\,\M}=\bangle{\M,\,|\mathsf T|^k\,\varphi}\,,\\[4pt]
\tilde\lambda_k:=\left\||v\cdot\xi|^k\,\psi\,\M\right\|_{-\eta}\,,\\[4pt]
\tilde\mu_k:=\left\||v\cdot\xi|^k\,\varphi\,\M\right\|_{-\eta}\,.
\end{array}
\ee
Notice that only the case $\eta = - \beta$ will be needed if $\gamma>\beta$. When $\gamma\le \beta$, we shall assume that $\eta\in(-\,\gamma,0)$. See Section~\ref{Sec:beta+gamma<0} for consequences.

%%%%%%%%%%%%%%%%%%%%%%%%%%%%%%%%%%%%%%%%%%%%%%%%%%%%%%%%%%%%%%%%%%%%%%
\subsection{Quantitative estimates of \texorpdfstring{$\mu_\mathsf L$ and $\lambda_\mathsf L$}{mu L and lambda L}}\label{Sec:mu-lambda}

%---------------------------------------------------------------------
\begin{proposition}\label{prop:entropy_decay_A_Assumption} Under Assumption~{\rm (H)}, if $\eta\in(-\,\gamma,\gamma)$ is such that $\eta\ge- \beta$, then $\lambda_\mathsf L$ is finite and
\[
\forall\,\xi\in\R^d\,,\quad\mu_\mathsf L(\xi)\lesssim\frac{|\xi|^{\alpha }}{\bangle\xi^{\alpha }}\,.
\]
\end{proposition}
%---------------------------------------------------------------------

%.....................................................................
\subsubsection{Generalized Fokker-Planck operators}\label{Sec:Fokker-Planck}

%---------------------------------------------------------------------
\begin{lemma} With $\mathsf L=\mathsf L_1$, we have
\[
\mu_\mathsf L\lesssim|\xi|^{\min\left\{2,\frac{\gamma+4+\eta}3\right\}}\,\mathds1_{|\xi|\le1}+|\xi|^{-2}\,\mathds1_{|\xi|\ge1}\quad\mbox{and}\quad\lambda_\mathsf L\lesssim 1\,.
\]
\end{lemma}
%---------------------------------------------------------------------
We recall that $\mathsf L_1$ is self-adjoint.
\begin{proof} Let us start by estimating $\mu_\mathsf L$. With
\begin{align*}
\M^{-1}\,\mathsf L\big((v\cdot\xi)\,\varphi\,\M\big)&=\nabla_v\cdot\Big(\M\,\nabla_v\,\big((v\cdot\xi)\,\varphi\big)\Big)\\
&=\Delta_v\big((v\cdot\xi)\,\varphi\big)-(d+\gamma)\,\frac v{\bangle v^2}\cdot\nabla_v\big((v\cdot\xi)\,\varphi\big)
\end{align*}
and $\nabla_v\big((v\cdot\xi)\,\varphi\big)=\varphi\,\xi+(v\cdot\xi)\,\nabla_v\varphi$, $\Delta_v\big((v\cdot\xi)\,\varphi\big)=2\,\xi\cdot\nabla_v\varphi+(v\cdot\xi)\,\Delta_v\varphi$, we end up with
\[
\M^{-1}\,\mathsf L\big((v\cdot\xi)\,\varphi\big)=2\,\xi\cdot\nabla_v\varphi+(v\cdot\xi)\(\Delta_v\varphi-\frac{(d+\gamma)}{\bangle v^2}\,\big(\varphi+v\cdot\nabla_v\varphi\big)\)\,.
\]
We recall that $\beta=2$ and
\[
\varphi(\xi,v)=\frac{\bangle v^2}{A(\xi,v)}\quad\mbox{where}\quad A(\xi,v):=1+\bangle v^6 |\xi|^2\,,
\]
so that
\[
\nabla_v\varphi=\(2\,A^{-1}-6\bangle v^6|\xi|^2A^{-2}\)v=2\(1-2\bangle v^6|\xi|^2\)\frac v{A^2}
\]
and
\[
\xi\cdot\nabla_v\varphi=2\(1-2\bangle v^6|\xi|^2\)\frac{v\cdot\xi}{A^2}\,,\quad v\cdot\nabla_v\varphi=2\(1-2\bangle v^6|\xi|^2\)\frac{|v|^2}{A^2}\,.
\]
Using $\bangle v^6|\xi|^2\le A$, we can readily estimate
\[
\left|\xi\cdot\nabla_v\varphi\right|\lesssim|v\cdot\xi| A^{-1},\quad
\left|v\cdot\nabla_v\varphi \right|\lesssim \bangle v^2 A^{-1}\,.
\]
The last part to estimate is
\begin{align*}
\Delta_v\varphi&=2\(1-2\bangle v^6|\xi|^2\)\nabla_v\cdot\(\frac v{A^2}\)+2\,\nabla_v\cdot\(1-2\bangle v^6|\xi|^2\)\frac v{A^2}\\
&=\frac2{A^2}\(1-2\bangle v^6|\xi|^2\)\(d+12\,|v|^2\bangle v^4|\xi|^2\,A^{-1}\)-24\,|v|^2\bangle v^4|\xi|^2\,A^{-2}\,,
\end{align*}
from which we deduce that
\[
\left|\Delta_v\varphi\right|\lesssim A^{-1}\,.
\]
Combining previous estimates, we thus end up with
\begin{align*}
\left|\M^{-1}\,\mathsf L\big((v\cdot\xi)\,\varphi\big)\right|&\lesssim|v\cdot\xi|\,A^{-1}\,.
\end{align*}
This provides us with the estimate
\[
\mu_\mathsf L(\xi)=\nrm{\mathsf L\((v\cdot\xi)\,\varphi(\xi,\cdot)\,\M\)}{\mathrm L^2(\bangle v^{-\eta}\d\mu)}^2\lesssim\rint\frac{|v\cdot\xi|^2}{\big(1+\bangle v^6|\xi|^2\big)^2}\frac{\d v}{\bangle v^{d+\gamma+\eta}}
\]
which allows us to conclude by elementary computations. Similar computations will be detailed in the proof of Lemma~\ref{lem:equiv_symbol_mu}.

Next we have to estimate $\lambda_\mathsf L$. After recalling that $\psi=\bangle v^{-2}$, we observe that
\[
\left|\M^{-1}\,\mathsf L\(\psi\,\M\)\right|=\Big|\Delta_v\psi-(d+\gamma)\,\frac v{\bangle v^2}\cdot\nabla_v\psi\Big|
\]
is a bounded quantity. Since $\bangle v^{-\eta}\M\in\mathrm L^1(\R^d)$, we conclude that $\lambda_\mathsf L$ is bounded.~$\square$\end{proof}

%.....................................................................
\subsubsection{Scattering collision operators}\label{Sec:Scattering}

%---------------------------------------------------------------------
\begin{lemma} Assume that~{\rm (H)} holds. With $\mathsf L=\mathsf L_2$, we have
\[
\mu_\mathsf L\lesssim|\xi|^{\min\left\{1,1+\frac{\gamma+\eta-2}{|1+ \beta|}\right\}}\,\mathds1_{|\xi|<1}+|\xi|^{-2}\,\mathds1_{|\xi|\ge1}\quad\mbox{and}\quad\lambda_\mathsf L\lesssim 1\,.
\]
\end{lemma}
%---------------------------------------------------------------------
\begin{proof} To estimate $\mu_\mathsf L$, we write
\begin{align*}
\M^{-1}\,\mathsf L^*\((v\cdot\xi)\,\varphi\,\M\)&=\rint\mathrm b(v',v)\,\big((v'\cdot\xi)\,\varphi(v')-(v\cdot\xi)\,\varphi(v)\big)\,\M(v')\,\d v'\\
&=\rint\mathrm b(v',v)\,(v'\cdot\xi)\,\varphi(v')\,\M(v')\,\d v'-(v\cdot\xi)\,\varphi(v)\,\nu(v).
\end{align*}
The Cauchy-Schwarz inequality yields
\begin{align*}
&\rint\left|\rint\mathrm b(v',v)\,(v'\cdot\xi)\,\varphi(v')\,\M'\,\d v'\right|^2 \bangle v^{-\eta}\M\,\d v\\
&\le\rint\(\rint\big|(v'\cdot\xi)\,\varphi(v')\big|^2\bangle v^{- \beta}\M'\,\d v'\)\(\rint\frac{\mathrm b(v',v)^2}{\nu(v')}\M'\,\d v'\)\bangle v^{-\eta}\M\,\d v\\
&\le\mathcal C_{\mathrm b}\rint\big|\nu(v)\,(v\cdot\xi)\,\varphi(v)\big|^2\bangle v^{-\eta}\M\,\d v\,,
\end{align*}
where, by Assumption~\eqref{hyp:b_bounds2}, $\mathcal C_{\mathrm b}=\riint\frac{\mathrm b(v',v)^2}{\nu(v')\,\nu(v)}\M\M'\,\d v\,\d v'$ is finite. Hence
\[
\rint\big|\nu(v)\,(v\cdot\xi)\,\varphi\big|^2 \bangle v^{-\eta}\M\,\d v\le C\rint\frac{|v\cdot\xi|^2}{\big(1+\bangle v^{2\,|1 + \beta|}\,|\xi|^2\big)^2}\frac{\d v}{\bangle v^{d+\gamma+\eta}}
\]
for some positive constant $C$, which provides us with the result.

The estimate for $\lambda_\mathsf L$ arises from
\begin{align*}
\M^{-1}\,\mathsf L^*\(\psi\,\M\)&=\rint\mathrm b(v',v)\,\big(\psi (v')-\psi(v)\big)\,\M(v')\,\d v'\\
&=\rint\mathrm b(v',v)\,\psi (v')\,\M(v')\,\d v'-\nu(v)\,\psi(v)\,.
\end{align*}
Again, the Cauchy-Schwarz inequality yields
\begin{align*}
\left|\rint\mathrm b(v',v)\,\psi(v')\,\M'\,\d v'\right|^2&\le\rint|\psi(v')|^2\bangle{v'}^{- \beta}\M'\,\d v'\rint\frac{\mathrm b(v',v)^2}{\nu(v')}\M'\,\d v'\\
&\le\mathcal C_{\mathrm b}\left|\nu(v)\,\psi(v)\right|^2\,,
\end{align*}
so that $\left|\M^{-1}\,\mathsf L^*\(\psi\,\M\)\right|\le\big(\mathcal C_{\mathrm b}^{1/2}+1\big)\,\left|\nu(v)\,\psi(v)\right|$. It follows from
\[
\rint|\nu(v)\,\psi(v)|^2 \bangle v^{-\eta}\M\,\d v\le C \rint\frac{\d v}{\bangle v^{d+\gamma+\eta+2\,\beta+4}}
\]
that $\lambda_\mathsf L\lesssim 1$ because $\gamma+\eta + 2\, \beta+4>\gamma + \beta+\eta +\beta>0$.
\Qed\end{proof}

%.....................................................................
\subsubsection{Fractional Fokker-Planck operators}\label{Sec:FractionalFokker-Planck}

%---------------------------------------------------------------------
\begin{lemma}\label{lem:estim_lap} For any $\sigma\in(0,2)$, we have
\[
|\Delta^{\sigma/2}\big((v\cdot\xi)\,\varphi\big)|\lesssim|\xi|^{\frac{\alpha }2}\,\mathds1_{|\xi|\le1}+\mathds1_{|\xi|\ge1}\,.
\]
\end{lemma}
%---------------------------------------------------------------------
\begin{proof} Let us introduce the notation
\[
\forall\,v\in\R^d,\quad m(v):=(v\cdot\xi)\,\varphi(v)
\]
and estimate the fractional Laplacian by $I_1+I_2$ where
\begin{align*}
&I_1:=\int_{|v-v'|<\bangle v/2}\frac{m(v')-m(v)-(v'-v)\cdot\nabla m(v)}{|v-v'|^{d+\sigma}}\,\d v'\,,\\
&I_2:=\int_{|v-v'|\ge\bangle v/2}\frac{m(v')-m(v)}{|v-v'|^{d+\sigma}}\,\d v'\,.
\end{align*}

%.....................................................................
\step{Step 1 : a bound of $I_1$.}

We perform a second order Taylor expansion. From
\begin{align*}
\nabla_v\varphi&=\(\beta+(\beta-2\,| 1+\beta |)\bangle v^{2\,| 1+ \beta |}|\xi|^2\)\bangle v^{- \beta-2} \varphi^2\,v\,,
\end{align*}
we deduce that $|\nabla_v\varphi|\lesssim \bangle v^{-1}\varphi$. In order to estimate the Hessian of~$\varphi$, we write
\begin{align*}
\left|\nabla_v^2\varphi(v)\right|&=\left|\nabla_v\(\(\beta+(\beta-2\,|1+ \beta|) \bangle v^{2\,|1+ \beta|}|\xi|^2\)\bangle v^{- \beta-2} \varphi^2\,v\)\right|\\
&\lesssim \bangle v^{-2} \varphi^2+\bangle v^{-1}\left|\nabla_v\varphi \right|\,,
\end{align*}
from which we deduce that $|\nabla^2_v\varphi|\lesssim\bangle v^{-2}\varphi$. It turns out that
\[
\left|\nabla_v^2\big((v\cdot\xi)\,\varphi\big)\right|\lesssim\left|\nabla_v\varphi (v)\right||\xi|+|v\cdot\xi|\left|\nabla^2_v\(\varphi(v)\)\right|\lesssim|\xi|\bangle v^{-1}\varphi
\]
because $\nabla_v\big((v\cdot\xi)\,\varphi(v)\big)=\varphi(v)\,\xi+(v\cdot\xi)\,\nabla_v\varphi(v)$.
Therefore,
\begin{multline*}
\left|I_1\right|\le\int_{|z|\le\bangle v/2}\frac{\nrm{\nabla_v^2 m}{\mathrm L^\infty(B(v,\bangle v/2))}}{|z|^{d+\sigma-2}}\,\d z\\
\le\frac{2^{\sigma-2}\,\omega_d\,|\xi|}{(2-\sigma)\bangle v^{\sigma-2}}\,\nrm{\bangle{\cdot}^{-1}\varphi}{\mathrm L^\infty(B(v,\bangle v/2))}\lesssim\frac{|\xi|\,\varphi(v)}{(2-\sigma)\bangle v^{\sigma-1}}
\end{multline*}
because $\bangle{v'}$ is comparable to $\bangle v$ uniformly on $B(v,\bangle v/2)$.

%.....................................................................
\step{Step 2: a bound of $I_2$.}

We distinguish two cases, $|\xi|\le1$ and $|\xi|\ge1$.

\medskip\noindent$\bullet$ Assume that $|\xi|\ge1$. We estimate $I_2$ by the three integrals
\begin{multline*}
\int_{\substack{|v-v'|\ge\bangle v/2\\ |v'|<\bangle v}}\frac{\left|m(v')\right|}{|v-v'|^{d+\sigma}}\,\d v'\,,\quad\int_{\substack{|v-v'|\ge\bangle v/2\\ |v'|\ge\bangle v}}\frac{\left|m(v')\right|}{|v-v'|^{d+\sigma}}\,\d v'\\
\mbox{and}\quad\int_{|v-v'|\ge\bangle v/2}\frac{\left|m(v)\right|}{|v-v'|^{d+\sigma}}\,\d v'
\end{multline*}
so that
\begin{align*}
|I_2|&\le\frac{2^{d+\sigma}}{\bangle v^{d+\sigma}}\,\nrm{m}{\mathrm L^1(B_0(\bangle v))}+\frac{2^\sigma\omega_d\,\nrm{m}{\mathrm L^\infty(B_0^c(\bangle v))}}{\sigma\bangle v^\sigma}+\frac{2^\sigma\omega_d\left|m(v)\right|}{\sigma\bangle v^\sigma}\\
&\lesssim \bangle v^{-\sigma}\(\bangle v^{-d}\,\nrm{m}{\mathrm L^1(B_0(\bangle v))}+\nrm{m}{\mathrm L^\infty(B_0^c(\bangle v))}+\left|m(v)\right|\).
\end{align*}
To proceed further, we have to estimate $\bangle v^{-d}\,\nrm{m}{\mathrm L^1(B_0(\bangle v))}$ and $\nrm{m}{\mathrm L^\infty(B_0^c(\bangle v))}$ for any $v\in\R^d$ and this is where $|\xi|\ge1$ will help. Let us observe that
\[
\bangle v^{-d} \nrm m{\mathrm L^1(B_0(\bangle v))}\le\left\{\begin{array}{ll}
\bangle v^{-d} |\xi|^{-1}&\mbox{ if }- \beta+2\,|1+ \beta| -1>d\,,\\[2pt]
\bangle v^{-d}\bangle v^{d+1+ \beta-2\,|1+ \beta|}|\xi|^{-1}&\mbox{ if }- \beta+2\,|1+ \beta|-1<d\,,\\[2pt]
2\,|\xi|\bangle v\varphi(v)&\mbox{ if }- \beta+2\,|1+ \beta|-1<d\,.
\end{array}\right.
\]
For any $v'\in B_0^c(\bangle v)$, we have
\begin{multline*}
|(v'\cdot\xi)\,\varphi(v')|\le\bangle{v'}|\xi|\frac{\bangle{v'}^{\beta}}{1+\bangle{v'}^{2\,|1+ \beta|}|\xi|^2}
\le\frac{\bangle{v'}^{1+ \beta}|\xi|}{1+\bangle{v'}^{2\,|1+ \beta|}|\xi|^2}\\
\le\frac{\bangle v^{1+\beta}|\xi|}{1+\bangle v^{2\,|1+ \beta|}|\xi|^2}=\bangle v\,|\xi|\,\varphi(v)\,,
\end{multline*}
where we have used that $\bangle{v'} \mapsto\frac{\bangle{v'}^{1+\beta}|\xi|}{1+\bangle{v'}^{2\,|1+\beta|}|\xi|^2}$ is decreasing for $\bangle{v'}\ge\bangle v$. Indeed, when $1+\beta\le0$ this is straightforward and when $1+\beta\ge 0$ it results from the fact that $\bangle{v'}^{|1+\beta|}|\xi|\ge1$ because $|\xi|\ge1$. Hence
\[
|I_2|\lesssim\left\{\begin{array}{ll}
\bangle v^{-\sigma}\(\bangle v^{-d}\,|\xi|^{-1}+\bangle v\,|\xi|\,\varphi(v)\)\quad&\mbox{ if }- \beta+2\,|1+ \beta|-1>d\,,\\[2pt]
\bangle v^{-\sigma}\bangle v |\xi|\,\varphi(v)&\mbox{ if }- \beta+2\,|1+ \beta|-1<d\,.
\end{array}\right.
\]

\medskip\noindent$\bullet$ Assume now that $|\xi|<1$. Let us write
\begin{multline*}
\left| I_2 \right|\le\int_{|z|\ge\bangle v/2} \sup_{|v-v'|>\bangle v/2}\(\frac{|m(v)-m(v')|}{|v-v'|^\ell}\)\frac{\,\d z}{|z|^{d+\sigma-\ell}}\\
\le\frac{ 2^{\sigma-\ell}\,\omega_d}{(\sigma-\ell)\bangle v^{\sigma-\ell}}\,\sup_{|v-v'|>\bangle v/2}\(\frac{|m(v)-m(v')|}{|v-v'|^\ell}\)
\end{multline*}
where $\ell$ will be chosen later. The next step is to estimate the $\ell-$H\"older semi-norm of $m$. For $1+\beta >0$ and any $w\in\R^d$, we may write
\begin{multline*}
\vert m(w) \vert\le|\xi|\bangle w\varphi=\frac{\bangle w^{1+\beta}|\xi|}{1+\bangle w^{2\,|1+ \beta|}|\xi|^2}\\\le|\xi|^\frac{\alpha }2 \bangle w^\frac{\alpha (1+\beta)}2\frac{\bangle w^\frac{(1+\beta)(2-\alpha )}2\,|\xi|^\frac{2-\alpha }2}{1+\bangle w^{2\,|1+ \beta|}|\xi|^2}
\lesssim|\xi|^\frac{\alpha }2 \bangle w^\ell\,,
\end{multline*}
with $\ell=\alpha (1+\beta)/2\in(0,1)$. For any $(v,v')$ such that $|v-v'|>\bangle v/2$, we deduce that
\[
|m(v)-m(v')|\lesssim|\xi|^\frac{\alpha }2\(2\bangle v^\ell+|v'-v|^\ell\)\lesssim|\xi|^\frac{\alpha }2\,|v-v'|^\ell\,,
\]
and finally obtain
\[
\left| I_2 \right|\le|\xi|^\frac{\alpha }2\frac{ 2^{\sigma-\ell}\omega_d}{(\sigma-\ell)\bangle v^{\sigma-\ell}}\,.
\]
In the case $1+\beta \leq 0$, the estimate can be performed exactly as for $|\xi|\ge1$ and we do not repeat the argument.
\Qed\end{proof}

%---------------------------------------------------------------------
\begin{proposition}\label{prop:estcoeffmuFFP} Let $\gamma>|\beta|$. With $\mathsf L=\mathsf L_3$, we have
\[
\mu_\mathsf L\lesssim|\xi|^{\alpha }\,\mathds1_{|\xi|\le1}+\mathds1_{|\xi|\ge1}\,.
\]
\end{proposition}
%---------------------------------------------------------------------
\begin{proof} We recall that $\M^{-1}\,\mathsf L^*(\M\,\cdot)=\Delta_v^{\sigma/2}-E\,\cdot\nabla_v$ and compute
\begin{multline*}
\mu_{\mathsf L}^2=\rint\left| \Delta_v^{\sigma/2}\big((v\cdot\xi)\,\varphi\big)-E\,\cdot\nabla_v\big((v\cdot\xi)\,\varphi\big)\right|^2\frac{\,\d v}{\bangle v^{d+\gamma+\eta}}\\
\le2\rint\left|\Delta_v^{\sigma/2}\big((v\cdot\xi)\,\varphi\big)\right|^2\frac{\,\d v}{\bangle v^{d+\gamma+\eta}}\\
+2\rint\left|E\,\cdot\nabla_v\big((v\cdot\xi)\,\varphi\big)\right|^2\frac{\,\d v}{\bangle v^{d+\gamma+\eta}}\,.
\end{multline*}
We have to estimate the two integrals of the latter right-hand side. The first one follows from Lemma~\ref{lem:estim_lap}. As for the second one, using Proposition~\ref{prop:estE}, we obtain
\begin{align*}
\left|E\cdot\nabla_v\big((v\cdot\xi)\,\varphi\big)\right|&\lesssim\left|E\cdot\xi\,\varphi\right|+\left|(v\cdot\xi)\,E\cdot\nabla_v\varphi\right|\\
&\lesssim|v\cdot\xi|\bangle v^{- \beta}\varphi+|v\cdot\xi|\,\varphi\bangle v^{-2}\bangle v^{- \beta}|v|^2\\
&\lesssim|v\cdot\xi|\bangle v^{- \beta}\varphi\,,
\end{align*}
so that
\begin{multline*}
\nrm{E\cdot\nabla_v\big((v\cdot\xi)\,\varphi\big)}{\mathrm L^2(\bangle v^{-\eta}\M\,\d v)}^2\lesssim\rint\frac{|v\cdot\xi|^2}{(1+\bangle v^{2\,|1+\beta|}|\xi|^2)^2}\frac{\,\d v}{\bangle v^{d+\gamma+\eta}}\\\lesssim|\xi|^{\alpha +\frac{\beta+\eta}{1+\beta}}\,\mathds1_{|\xi|\le1}+\mathds1_{|\xi|\ge1}\,.
\end{multline*}
\Qed\end{proof}

%---------------------------------------------------------------------
\begin{proposition}\label{prop:estcoefflambdaFFP} Let $\gamma\ge (- \beta)_+$ and consider $\mathsf L=\mathsf L_3$. There exists a constant $C>0$, independent of $\xi$, such that
\[
\lambda_\mathsf L\le C\,.
\]
\end{proposition}
%---------------------------------------------------------------------
\begin{proof} We follow the same steps as in the proof of Proposition~\ref{prop:estcoeffmuFFP}. We have
\begin{align*}
\lambda_\mathsf L^2&=\rint\big(\Delta_v^{\sigma/2}\,\psi-E\,\cdot\nabla_v\psi\big)^2\,\frac{\,\d v}{\bangle v^{d+\gamma+\eta}}\\
&\le2\rint\big|\Delta_v^{\sigma/2}\,\psi\big|^2\,\frac{\,\d v}{\bangle v^{d+\gamma+\eta}}+2\rint\left|E\,\cdot\nabla_v\psi\right|^2\,\frac{\,\d v}{\bangle v^{d+\gamma+\eta}}\,.
\end{align*}
Since $\Delta_v^{\sigma/2}\psi$ is a bounded function, the first integral of the right-hand side is bounded because $\gamma+\eta>0$. For the second integral, we simply observe that
\[
\rint\left|E\,\cdot\nabla_v\psi\right|^2\,\frac{\,\d v}{\bangle v^{d+\gamma+\eta}}\le\rint\frac{|v|^2\,\d v}{\bangle v^{d+\gamma+\eta +2\,\beta+8}}
\]
is bounded because $\gamma+\eta + 2\,\beta+6 = (\gamma+\beta) + (\eta+\beta)+6>0$.
\Qed\end{proof}

%%%%%%%%%%%%%%%%%%%%%%%%%%%%%%%%%%%%%%%%%%%%%%%%%%%%%%%%%%%%%%%%%%%%%%
\subsection{Quantitative estimates of \texorpdfstring{$\mu_2$ and $\tilde\mu_k$}{mu2 and muk}}\label{Sec::equiv_symbol_mu_and_lambda}

We recall that the coefficient $\mu_2$ and $\tilde\mu_k$ have been defined in~\eqref{Notations:LambdaMu}. The coefficients $\tilde\mu_1$ and $\tilde\mu_2$ are well defined when $- \beta\le\eta<\gamma$ since $-( \beta + 1)+\left| \beta + 1\right|=2\,(- \beta-1)_+$ so that, for any $k\le2$,
\[
\eta+\gamma -2\, \beta+4\,|1+\beta|=\(\eta+ \beta\)+\(\gamma+\beta\)+8\,(- \beta-1)_++4>2\,k\,.
\]
The notation $a\simeq b$ means that there exists a constant $C>0$ such that $a/C\le b\le C\,a$. Our first result investigates the dependence of $\mu_2$ and $\tilde\mu_k$ in $\xi\in\R^d$.
%---------------------------------------------------------------------
\begin{lemma}\label{lem:equiv_symbol_mu}
For $- \beta\le\eta<\gamma$ with $\gamma>0$, the coefficient $\mu_2$ is bounded from above and below for large values of $|\xi|$ and satisfies
\begin{align*}
\mu_2(\xi)&\underset{\xi\to 0}\simeq|\xi|^{\min\left\{2,2+\frac{\gamma - \beta - 2}{|1+\beta|}\right\}}&\quad\mbox{if}\quad \gamma\neq2+ \beta\,,\\
\mu_2(\xi)&\underset{\xi\to 0}\sim-\,\frac1{d\left|1 +\beta\right|}\,|\xi|^2\,\log|\xi|&\quad\mbox{if}\quad \gamma=2+\beta\,.
\end{align*}
If $\eta+\gamma - 2\,\beta+4\,|1+\beta|>2\,k$, then $\tilde\mu_k(\xi)\underset{|\xi|\to+\infty}\simeq|\xi|^{k-2}$ and
\begin{align*}
\tilde\mu_k(\xi)&\underset{\xi\to 0}\simeq|\xi|^{\min\left\{k,k+\frac{\gamma+\eta -2\,\beta-2k}{2\,|1+\beta|}\right\}}&\quad\mbox{if}\quad\gamma - 2\,\beta+\eta\neq2\,k\,,\\
\tilde\mu_k(\xi)&\underset{\xi\to 0}\simeq-\,|\xi|^k\,\log|\xi|&\quad\mbox{if}\quad\gamma - 2\, \beta+\eta=2\,k\,.
\end{align*}
\end{lemma}
%---------------------------------------------------------------------
\begin{proof} We start by considering $\xi\to0$. Let $c:=|1+\beta|\ge 0$. If $ \gamma>2+\beta$, then
\[
\mu_2(\xi)\underset{\xi\to 0}\sim|\xi|^2\rint\frac{c_\gamma\left|v_1\right|^2}{\bangle v^{d+\gamma- \beta}}\,\d v\,.
\]
$\bullet$ If $\gamma<2+\beta$, then $1+ \beta >0$ and $c>0$. With the change of variables $v=u\,|\xi|^{-1/c}$, we observe that
\[
\mu_2(\xi)\underset{\xi\to 0}\sim|\xi|^{2+\frac{\gamma- \beta-2}c}\rint\frac{|u_1|^2 }{ 1+|u|^{2\,c}}\frac{c_\gamma\,\d u}{ |u|^{d+\gamma- \beta}}
\]
using $\bangle{u\,|\xi|^{-1/c}}\underset{\xi\to 0}\sim|u|\,|\xi|^{-1/c}$ for any $u\in\R^d\backslash\lbrace0\rbrace$.
\\
$\bullet$ If $\gamma=2+\beta$ and $\gamma + \beta>0$, then $1+ \beta>0$, $c=1+\beta$ is positive and
\begin{align*}
\mu_2=|\xi|^2\rint\frac{c_\gamma\left|v_1\right|^2}{1+\bangle v^{2\,c}|\xi|^2}\frac{\d v}{\bangle v^{d+2}}\,.
\end{align*}
With the change of variables $v=u\,|\xi|^{-1/c}$, we have that
\begin{align*}
I_0:=\int_{|v|\ge|\xi|^{-\frac1c}}\frac{|v_1|^2}{1+\bangle v^{2\,c}|\xi|^2}\frac{\d v}{\bangle v^{d+2}}\underset{\xi\to 0}\sim\int_{|u|\ge1}\frac{|u_1|^2}{1+|u|^{2\,c}}\frac{\d u}{|u|^{d+2}}
\end{align*}
is finite. Using the invariance under rotation with respect to $\xi\in\R^d$,
\[
d\int_{|v|<|\xi|^{-\frac1c}}\frac{|v_1|^2}{1+\bangle v^{2\,c}|\xi|^2}\frac{\d v}{\bangle v^{d+2}}=\int_{|v|<|\xi|^{-\frac1c}}\frac{|v|^2}{1+\bangle v^{2\,c}|\xi|^2}\frac{\d v}{\bangle v^{d+2}}
\]
can be splitted, using $|v|^2=\bangle v^2 -1$, into
\[
-\int_{|v|<|\xi|^{-\frac1c}}\frac1{1+\bangle v^{2\,c}|\xi|^2}\frac{\d v}{\bangle v^{d+2}}\underset{\xi\to 0}\sim-\rint\frac{\d v}{\bangle v^{d+2}}
\]
and
\[
\int_{|v|<|\xi|^{-\frac1c}}\frac1{1+\bangle v^{2\,c}|\xi|^2}\frac{\d v}{\bangle v^d}=\int_{|v|<|\xi|^{-\frac1c}}\frac{\d v}{\bangle v^d}-\int_{|v|<|\xi|^{-\frac1c}}\frac1{1+\bangle v^{-2\,c}|\xi|^{-2}}\frac{\d v}{\bangle v^d}
\]
using $\frac1{1+X}=1-\frac1{1+1/X}$ with $X=\bangle v^{2\,c}|\xi|^2$.
\[
\int_{|v|<|\xi|^{-\frac1c}}\frac{\d v}{\bangle v^d}\underset{\xi\to 0}\sim-\,\frac{\omega_d}c\,\log|\xi|
\]
and
\[
\int_{|v|<|\xi|^{-\frac1c}}\frac1{1+\bangle v^{-2\,c}|\xi|^{-2}}\frac{\d v}{\bangle v^d}\underset{\xi\to 0}\sim\int_{|u|<1}\frac{|u|^{2c}}{1+|u|^{2c}}\frac{\d u}{|u|^d}
\]
by the change of variables $v=|\xi|^{-1/c}\,u$. After collecting terms, this yields
\[
\mu_2\underset{\xi\to 0}\sim-\,\frac{\omega_d}{c\,d}\,|\xi|^2\,\log|\xi|\,.
\]

On the other hand, when $|\xi|\to+\infty$, we have
\begin{align*}
\mu_2(\xi)=\rint\frac{|v\cdot\xi|^2\bangle v^{- \beta}}{\bangle v^{-2\, \beta}+\bangle v^2\,|\xi|^2}\frac{c_\gamma\,\d v}{\bangle v^{d+\gamma}}\underset{|\xi|\to+\infty}\sim\rint\frac{c_\gamma\left|v_1\right|^2}{\bangle v^{d+\gamma+2+\beta}}\,\d v\,.
\end{align*}
The claim on $\mu_2$ is now completed. All other estimates follow from similar computations and we shall omit further details.
\Qed\end{proof}

The coefficients $\lambda_0$, $\lambda_1$, $\tilde\lambda_0$ and $\tilde\lambda_1$ have also been defined in~\eqref{Notations:LambdaMu}. Our second technical estimate goes as follows.
%---------------------------------------------------------------------
\begin{lemma}\label{lem:equiv_symbol_lambda} The coefficients $\lambda_0$ and $\lambda_1$ are well defined for any $\gamma>0$. The coefficients $\tilde\lambda_0$ and $\tilde\lambda_1$ are also well defined if $\gamma>0$ and $\eta>-\,\gamma$. \end{lemma}
%---------------------------------------------------------------------
The proof is straightforward and left to the reader.

%%%%%%%%%%%%%%%%%%%%%%%%%%%%%%%%%%%%%%%%%%%%%%%%%%%%%%%%%%%%%%%%%%%%%%
\subsection{A macroscopic coercivity estimate}\label{Sec:macro}

We recall that $\mathsf R_\xi[\Ff]=-\,\dt\re\,\langle\mathsf A_\xi\Ff,\Ff\rangle$ if $f$ solves~\eqref{eq:main}
where
\[
\mathsf A_\xi=\frac1{\bangle v^2}\,\mathsf\Pi\,\frac{\(-\,i\,v\cdot\xi\)\bangle v^{\beta}}{1+\bangle v^{2\,|1+\beta|}\,|\xi|^2}=\psi\,\mathsf\Pi\mathsf T^*\,\varphi\,\Ff
\]
with $\varphi(\xi,v)=\frac{\bangle v^{ \beta}}{1+\bangle v^{2\,|1+\beta|}\,|\xi|^2}$ and $\psi(v):=\bangle v^{-2}$. In this section, our goal is to establish an estimate of $\mathsf R_\xi[\Ff]$. In this section, we use the notation~\eqref{Notations:LambdaMu} and prove Proposition~\ref{prop:entropy_decay_A}.

\begin{proof}[Proof of Proposition~\ref{prop:entropy_decay_A}] Since $\varphi$ and $\psi $ commute with $\mathsf T$ and $\mathsf\Pi\mathsf T\mathsf\Pi=0$, we get
\begin{align*}
\mathsf A_\xi\,\mathsf\Pi=-\,\psi\,\mathsf\Pi\mathsf T\mathsf\Pi\,\varphi=0,\quad\mathsf A_\xi^*\mathsf T\mathsf\Pi=\varphi\,\mathsf T\mathsf\Pi\mathsf T\mathsf\Pi\,\psi=0\,.
\end{align*}
Moreover, $\mathsf L\mathsf\Pi=0$. With these identities, using the micro-macro decomposition $\Ff=\mathsf\Pi\Ff+(1-\mathsf\Pi)\Ff$, we find that
\[
\mathsf R_\xi[\Ff]:=\I1+\I2+\I3+\I4+\I5+\I6+\I7
\]
where
\[
\I1:=\bangle{\mathsf A_\xi\,\mathsf T \mathsf\Pi\Ff ,\mathsf\Pi\Ff}\,,\!\quad\I2:=\bangle{\mathsf A_\xi\,\mathsf T \mathsf\Pi\Ff ,(1-\mathsf\Pi)\Ff}\,,\!\quad\I3:=\bangle{\mathsf A_\xi\,\mathsf T (1-\mathsf\Pi)\Ff ,\mathsf\Pi\Ff}\,,
\]
\[
\I4:=\bangle{\mathsf A_\xi\,\mathsf T (1-\mathsf\Pi)\Ff ,(1-\mathsf\Pi)\Ff}\,,\quad\I5:=\bangle{\mathsf A_\xi (1-\mathsf\Pi)\Ff, \mathsf T (1-\mathsf\Pi)\Ff}\,,
\]
\[
\I6:=-\bangle{\mathsf A_\xi\,\mathsf L(1-\mathsf\Pi)\Ff,\Ff}\,,\quad\I7:=-\bangle{\mathsf A_\xi (1-\mathsf\Pi)\Ff,\mathsf L(1-\mathsf\Pi)\Ff}\,.
\]

%.....................................................................
\step{Step 1: macroscopic coercivity.}

Since $\irdv\M=1$ and $\mathsf\Pi\Ff(\xi,v)=\rho_{\Ff}(\xi)\,\M(v)$, we first notice that
\[
|\rho_{\Ff}(\xi)|^2=\rint|\rho_{\Ff}(\xi)\,\M|^2\,\d\mu=\n{\mathsf\Pi\Ff}^2\,.
\]
and
\[
\I1=\bangle{\mathsf A_\xi\,\mathsf T \M,\M}|\rho_{\Ff}|^2=\bangle{\psi\,\mathsf\Pi\,|\mathsf T|^2\,\varphi\,\M,\M}\,\n{\mathsf\Pi\Ff}^2=\lambda_0\,\mu_2\,\n{\mathsf\Pi\Ff}^2
\]
by definition of $\mathsf A_\xi$.

%.....................................................................
\step{Step 2: micro-macro terms.}

By definition of $\mathsf A_\xi$,
\[
\I2=\big\langle{\psi\,\mathsf\Pi\mathsf T^*\,\varphi\,\mathsf T\mathsf\Pi\Ff,(1-\mathsf\Pi)\Ff}\big\rangle=\bangle{\M,|\mathsf T|^2\,\varphi\,\M }\,\rho_{\Ff}\,\big\langle\psi\,\M,(1-\mathsf\Pi)\Ff\,\big\rangle
\]
can be estimated using $|\rho_{\Ff}|=\n{\mathsf\Pi\Ff}$ and the Cauchy-Schwarz inequality
\[
\big\langle\psi\,\M,(1-\mathsf\Pi)\Ff\,\big\rangle\le\nrm{\psi\,\M}{-\eta}\,\nrm{(1-\mathsf\Pi)\Ff}\eta
\]
by
\[
\left|\I2\right|\le\tilde\lambda_0\,\mu_2\,\n{\mathsf\Pi\Ff}\,\nrm{(1-\mathsf\Pi)\Ff}\eta\,.
\]
By similar estimates, we obtain

\begin{align*}
\left|\I3\right|&\le\lambda_0\,\tilde\mu_2\,\n{\mathsf\Pi\Ff}\,\nrm{(1-\mathsf\Pi)\Ff}\eta\,,\\
\left|\I4\right|&\le\tilde\lambda_0\,\tilde\mu_2\,\nrm{(1-\mathsf\Pi)\Ff}\eta^2\,,\\
\left|\I5\right|&\le\tilde\lambda_1\,\tilde\mu_1\,\nrm{(1-\mathsf\Pi)\Ff}\eta^2\,.
\end{align*}
To get a bound on $\I6$, we use the fact that $\mathsf T^*=-\mathsf T$ to obtain
\begin{align*}
\I6&=\bangle{\psi\,\mathsf\Pi\mathsf T\varphi\,\mathsf L(1-\mathsf\Pi)\Ff,\Ff}=\bangle{\M, \mathsf T\varphi\,\mathsf L(1-\mathsf\Pi)\Ff} \bangle{\M\,\psi,\Ff}\\
&=-\bangle{\mathsf L^*\mathsf T\,\varphi\,\M, (1-\mathsf\Pi)\Ff} \bangle{\M\,\psi,\Ff}\,.
\end{align*}
By the micro-macro decomposition $\Ff=\mathsf\Pi\Ff+(1-\mathsf\Pi)\Ff$, we have
\[
\bangle{\M\,\psi,\Ff}=\lambda_0\,\rho_{\Ff}+\bangle{\M\,\psi,(1-\mathsf\Pi)\Ff}
\]
and the Cauchy-Schwarz inequality gives
\[
\bangle{\mathsf L^*\mathsf T\,\varphi\,\M, (1-\mathsf\Pi)\Ff}\le\mu_\mathsf L\,\nrm{(1-\mathsf\Pi)\Ff}\eta\,.
\]
This yields
\[
\left|\I6\right|\le\lambda_0\,\mu_\mathsf L\,\n{\mathsf\Pi\Ff}\,\nrm{(1-\mathsf\Pi)\Ff}\eta+\tilde\lambda_0\,\mu_\mathsf L\,\nrm{(1-\mathsf\Pi)\Ff}\eta^2\,.
\]
In the same way, we get
\begin{align*}
\left|\I7\right|&\le\lambda_\mathsf L\,\tilde\mu_1\,\nrm{(1-\mathsf\Pi)\Ff}\eta^2\,.
\end{align*}

%.....................................................................
\step{Step 3: cross terms.}

With $X:=\n{\mathsf\Pi\Ff}$ and $Y:=\nrm{(1-\mathsf\Pi)\Ff}\eta$, we collect all above estimates into
\begin{align*}
\mathsf R_\xi[\Ff]\le-\lambda_0\,\mu_2\,X^2&+\(\tilde\lambda_0\,\mu_2+\lambda_0\,\tilde\mu_2+\lambda_0\,\mu_\mathsf L\)XY\\
&+\(\tilde\lambda_0\,\tilde\mu_2+\tilde\lambda_1\,\tilde\mu_1+\tilde\lambda_0\,\mu_\mathsf L+\lambda_\mathsf L\,\tilde\mu_1\)Y^2\,,
\end{align*}
which by Young's inequality leads to
\begin{multline*}
\mathsf R_\xi[\Ff]\le\(\frac a2\(\tilde\lambda_0\,\mu_2+\lambda_0\,\tilde\mu_2+\lambda_0\,\mu_\mathsf L\)-\lambda_0\,\mu_2\)X^2\\
+\(\tilde\lambda_0\,\tilde\mu_2+\tilde\lambda_1\,\tilde\mu_1+\tilde\lambda_0\,\mu_\mathsf L+\lambda_\mathsf L\,\tilde\mu_1+\frac{\tilde\lambda_0\,\mu_2+\lambda_0\,\tilde\mu_2+\lambda_0\,\mu_\mathsf L}{2\,a}\)Y^2\,.
\end{multline*}
With the choice
\[
a=\frac{\lambda_0\,\mu_2}{\tilde\lambda_0\,\mu_2+\lambda_0\,\tilde\mu_2+\lambda_0\,\mu_\mathsf L}\,,
\]
we get
\[
\mathsf R_\xi[\Ff]\le-\,\frac12\,\lambda_0\,\mu_2\,\n{\mathsf\Pi\Ff}^2+\mathcal K(\xi)\,\nrm{(1-\mathsf\Pi)\Ff}\eta^2\,,
\]
with
\[
\mathcal K(\xi)=\tilde\lambda_0\,\tilde\mu_2+\tilde\lambda_1\,\tilde\mu_1+\tilde\lambda_0\,\mu_\mathsf L+\lambda_\mathsf L\,\tilde\mu_1+\frac{\(\tilde\lambda_0\,\mu_2+\lambda_0\,\tilde\mu_2+\lambda_0\,\mu_\mathsf L\)^2}{2\,\lambda_0\,\mu_2}\,.
\]

%.....................................................................
\step{Step 4: A uniform bound on $\mathcal K(\xi)$.}

According to Lemma~\ref{lem:equiv_symbol_lambda}, $\lambda_0$ and $\tilde\lambda_0$ are independent of $\xi$ and take finite positive values, so that
\[
\mathcal K(\xi)\lesssim \tilde\mu_2+\tilde\lambda_1\,\tilde\mu_1+\mu_\mathsf L+\lambda_\mathsf L\,\tilde\mu_1+\mu_2+\frac{\tilde\mu_2^2}{\mu_2}+\frac{\mu_\mathsf L^2}{\mu_2}\,.
\]
We also deduce from their definitions in~\eqref{Notations:LambdaMu} that $\mu_2$, $\tilde\mu_2$, $|\xi|\,\tilde\mu_1$ and $\tilde\lambda_1/|\xi|$ have finite, positive limits as $|\xi|\to+\infty$. By Proposition~\ref{prop:entropy_decay_A_Assumption}, $\mu_\mathsf L$ and $\lambda_\mathsf L$ are bounded from above, so that
\[
\forall\,\xi\in\R^d\;\mbox{such that}\;|\xi|\ge1\,,\quad\mathcal K(\xi)\lesssim 1+\mu_\mathsf L+\frac{\lambda_\mathsf L}{|\xi|}+\mu_\mathsf L^2\lesssim 1\,.
\]
It remains to investigate the behaviour of $\mathcal K(\xi)$ as $\xi\to0$ and we shall distinguish two main cases:
\\[4pt]
$\bullet$ if $1+\beta >0$, under the assumption that $\gamma\neq2+\beta$, $\gamma+\eta - 2\beta\neq4$ and $\gamma+\eta-2\beta\,\neq2$, for some positive constants $C_1$, $C_2$, $\widetilde C_1$, $\widetilde C_2$, we have
\[
\mu_2\underset{\xi\to 0}\sim C_2\,|\xi|^{\alpha }\,,\quad\tilde\mu_2\underset{\xi\to 0}\sim\widetilde C_2\,|\xi|^{\min\left\{2,\frac{\gamma+2\beta+\eta}{2\,(1+ \beta)}\right\}}\,,
\]
\[
\tilde\mu_1\underset{\xi\to 0}\sim C_1\,|\xi|^{\min\left\{1,\frac{\gamma+\eta}{2\,(1+\beta)}\right\}}\,,\quad\tilde\lambda_1\underset{\xi\to 0}\sim \widetilde C_1\,|\xi|\,,
\]
where $\alpha =\min\left\{2,\frac{\gamma+\beta}{1+\beta}\right\}$ as in~\eqref{alpha}. Since $\eta\ge- \beta\geq-2\beta-\gamma$ and $\eta + \gamma\ge 0$, this implies by Proposition~\ref{prop:entropy_decay_A_Assumption} that
\[
\forall\,\xi\in\R^d\;\mbox{such that}\;|\xi|\le1\,,\quad\mathcal K(\xi)\lesssim 1+\mu_\mathsf L+\lambda_\mathsf L+\frac{\tilde\mu_2^2}{\mu_2}+\frac{\mu_\mathsf L^2}{\mu_2}\lesssim 1+\frac{\tilde\mu_2^2}{\mu_2}\,.
\]
Then $\frac{\tilde\mu_2^2}{\mu_2}=O\(|\xi|^\varepsilon\)$ is bounded as $\xi\to0$ either if $\gamma<2+\beta$ because
\[
\textstyle\varepsilon=\min\left\{\frac{\eta +\beta}{1+\beta},\frac{4+3\beta-\gamma}{1+\beta}\right\}\,,\quad\eta\ge- \beta\,,\quad4+3\beta-\gamma=2\,(1+\beta)+2+(\beta-\gamma)\ge 0\,,
\]
or if $\gamma>2+\beta$ because
\[
\textstyle\varepsilon=\min\left\{2,\frac{\gamma+\eta-2}{1+\beta}\right\}\quad\mbox{and}\quad\eta\ge- \beta>2-\gamma\,.
\]
$\bullet$ if $1+\beta<0$, under the assumption that $\gamma-2\,\beta+\eta-4\neq 0$ and $\gamma -2\,\beta+\eta-2\neq 0$, for some positive constants $C_1$, $C_2$, $\widetilde C_1$, $\widetilde C_2$, we have
\[
\mu_2\underset{\xi\to 0}\sim C_2\,|\xi|^2\,,\quad\tilde\mu_2\underset{\xi\to 0}\sim\widetilde C_2\,|\xi|^{\min\left\{2,\frac{\gamma+\eta-6\,\beta-8}{2\,|1+\beta|}\right\}}\,,
\]
\[
\tilde\mu_1\underset{\xi\to 0}\sim C_1\,|\xi|^{\min\left\{1,\frac{\gamma+\eta - 4\,\beta-4}{2\,| \beta+1|}\right\}}\,,\quad\tilde\lambda_1\underset{\xi\to 0}\sim \widetilde C_1\,|\xi|\,.
\]
Since $\eta\ge- \beta>1$, we get $\gamma+\eta - 6\,\beta-8>0$ and $\gamma+\eta -4\, \beta-4>0$, so that
\[
\mathcal K(\xi)\lesssim 1+\frac{\tilde\mu_2^2}{\mu_2}\,.
\]
where $\frac{\tilde\mu_2^2}{\mu_2}=O\(|\xi|^\varepsilon\)$ is bounded as $\xi\to0$ because
\[
\textstyle\varepsilon=\min\left\{2,\frac{\gamma+\eta - 4\,\beta-6}{|\beta+1|}\right\}\quad\mbox{and}\quad\gamma+\eta - 4\,\beta-6>0\,.
\]
In the critical cases when a $\log|\xi|$ appears in the expression of $\mu_2$, $\tilde\mu_1$ or $\tilde\mu_2$, we obtain expressions of the form $|\xi|^\varepsilon\,\big|\log|\xi|\big|$ for some $\varepsilon>0$, so that all terms also remain bounded. We conclude that in all cases, $\mathcal K(\xi)$ is bounded from above uniformly with respect to $\xi$. This ends the proof of Proposition~\ref{prop:entropy_decay_A}.
\Qed\end{proof}

%%%%%%%%%%%%%%%%%%%%%%%%%%%%%%%%%%%%%%%%%%%%%%%%%%%%%%%%%%%%%%%%%%%%%%
\subsection{A fractional Nash inequality and consequences}\label{Sec:fractionalNash}

For any $a>0$, let us define the function $\mathcal L_a$ by
\[
\forall\,s\ge0\,,\quad\mathcal L_a(s):=\frac{s^a}{(1+s^2)^{a/2}}
\]
and the quadratic form
\[
\mathcal Q_a[u]:=\int_{\R^d}\mathcal L_a\big(|\xi|\big)\,|\hat u(\xi)|^2\,\d\xi
\]
where $\hat u$ denotes the Fourier transform of a function $u\in\mathrm L^2(\d x)$ given by
\[
\hat u(\xi)=(2\pi)^{-d/2}\irdx{e^{-\,i\,x\cdot\xi}\,u(x)}\,.
\]
We recall that by Plancherel's formula, $\nrm u{\mathrm L^2(\d x)}^2=\nrm{\hat u}{\mathrm L^2(\d\xi)}^2$.
%---------------------------------------------------------------------
\begin{lemma}\label{lem:fractionalNash} Let $d\ge1$ and $a\in(0,2]$. There is a monotone increasing function $\Phi_a:\R_+\to\R_+$ with $\Phi_a(s)\sim s^{d/(d+a)}$ as $s\to 0^+$ such that
\[
\forall\,u\in\mathcal D(\R^d)\,,\quad\nrm u{\mathrm L^2(\d x)}^2\le\nrm u{\mathrm L^1(\d x)}^2\,\Phi_a\(\frac{\mathcal Q_a[u]}{\nrm u{\mathrm L^1(\d x)}^2}\).
\]
\end{lemma}
%---------------------------------------------------------------------
\begin{proof} We rely on a simple argument based on Fourier analysis inspired by the proof of Nash's inequality in~\cite[page~935]{Nash58}, which goes as follows. Since $\nrm{\hat u}{\mathrm L^\infty(\d\xi)}\le\nrm u{\mathrm L^1(\d x)}$, we obtain
\begin{align*}
\nrm u{\mathrm L^2(\d x)}^2=\nrm{\hat u}{\mathrm L^2(\d\xi)}^2&\le\int_{|\xi|\le R}|\hat u(\xi)|^2\,\d\xi+\int_{|\xi|>R}|\hat u(\xi)|^2\,\d\xi\\
&\le\frac1d\,\omega_d\,\nrm u{\mathrm L^1(\d x)}^2\,R^d+\frac 1{\mathcal L_a(R)}\,\mathcal Q_a[u]
\end{align*}
for any $R>0$, using the monotonicity of $s\mapsto\mathcal L_a(s)$.

Let us consider the function
\[
f(x,R):=\frac1d\,R^d+\frac xa\,(1+R^{-2})^{a/2}
\]
and notice that, as a function of $R$, $f$ has a unique minimum $R=R(x)$ such that
\[
R^{d+a}\,(1+R^2)^{1-\frac a2}=x
\]
for any $x>0$. With $a\in(0,2]$, it is clear that $x\mapsto R(x)$ is monotone increasing and such that $R(x)\le x^{1/(d+2)}\,\big(1+o(1)\big)$ as $x\to+\infty$ and $R(x)=x^{1/(d+a)}\,\big(1+o(1)\big)$ as $x\to 0^+$. Altogether, for the optimal value $R=R(x)$, we obtain that $\phi(x)=f\big(x,R(x)\big)$ is such that
\begin{align*}
\phi(x)=\(\tfrac1d+\tfrac1a\)x^\frac d{d+a}\,\big(1+o(1)\big)\quad&\mbox{as}\quad s\to0^+\,,\\
\phi(x)=\frac xa\,\big(1+o(1)\big)\quad&\mbox{as}\quad s\to+\infty\,.
\end{align*}
The proof is concluded with $\Phi_a(s)=\omega_d\,\phi\big(\frac{a\,s}{\omega_d}\big)$.
\Qed\end{proof}

Let us consider the Fourier transform with respect to $x$ of a distribution function $f$ depending on $x$ and $v$ and define
\[
\mathsf Q_a[f]:=\irdmu{\mathcal Q_a[f]}\,.
\]
%---------------------------------------------------------------------
\begin{lemma}\label{lem:fractionalNash2} Let $d\ge1$ and $a\in(0,2]$. With the above notation, we have
\[
\forall\,f\in\mathrm L^1\cap\mathrm L^2(\d x\,\d\mu)\,,\quad\nrm{\mathsf\Pi f}{\mathrm L^2(\d x\,\d\mu)}^2\le\nrm f{\mathrm L^1(\d x\,\d v)}^2\,\Phi_a\(\frac{\mathsf Q_a[\mathsf\Pi f]}{\nrm f{\mathrm L^1(\d x\,\d v)}^2}\),
\]
where the function $\Phi_a$ is defined in Lemma~\ref{lem:fractionalNash}. \end{lemma}
%---------------------------------------------------------------------
\begin{proof} We apply the strategy of Lemma~\ref{lem:fractionalNash} to $\mathsf\Pi f=\rho_f\,\M$ and bound
\[
\nrm{\rho_f}{\mathrm L^2(\d x)}=\nrm{\mathsf\Pi f}{\mathrm L^2(\d x\,\d\mu)}^2=\nrm{\widehat{\mathsf\Pi f}}{\mathrm L^2(\d\xi\,d\mu)}^2
\]
by
\begin{equation*}
\iint_{|\xi|\le R}|\widehat{\mathsf\Pi f}(\xi,v)|^2\,\d\xi\,\d\mu=\frac{\omega_d}d\,R^d\,\nrm {\rho_f}{\mathrm L^1(\d x)}^2\irdmu{\M^2} =\frac{\omega_d}d\,R^d\,\nrm f{\mathrm L^1(\d x\,\d v)}
\end{equation*}
and
\[
\iint_{|\xi|>R}|\widehat{\mathsf\Pi f}(\xi,v)|^2\,\d\xi\,\d\mu\le\frac{\mathcal Q_a[\rho_f]}{\mathcal L_a(R)}\irdmu{\M^2}=\frac{\mathsf Q_a[\mathsf\Pi f]}{\mathcal L_a(R)}\,.
\]
{}From this point, the computations are exactly the same as in the proof of Lem\-ma~\ref{lem:fractionalNash}.\Qed\end{proof}

%%%%%%%%%%%%%%%%%%%%%%%%%%%%%%%%%%%%%%%%%%%%%%%%%%%%%%%%%%%%%%%%%%%%%%
\subsection{A limit case of the fractional Nash inequality}\label{Sec:fractionalNashLimit}

In the case when $\gamma=2+\beta$, we recall that
\[
\mathsf R_\xi[\Ff]\gtrsim\Lambda(\xi)\,\n{\mathsf\Pi\Ff}^2-\mathcal C\,\nrm{(1-\mathsf\Pi)\Ff}{- \beta}^2
\]
by Proposition~\ref{prop:entropy_decay_A}, where $\Lambda(\xi)=h(|\xi|)$ and $h(r)=r^2\,|\log r|/(1+r^2\,\log r)$. The function $h:[0,1/\sqrt e)\to\R$ is monotone increasing. Define
\[
\Phi(x):=\frac 1d\,x^{1+\frac2d} \left|\log x\right|\,.
\]
The proof of Corollary~\ref{Cor:fractionalNash2} relies on the following result.
%---------------------------------------------------------------------
\begin{lemma}\label{lem:fractionalNash3} Let $d\ge1$ and assume that $\gamma=2+\beta$. With the above notation, there exists a positive constant $\mathsf A$ such that, if
\[
\frac{\nrm{\mathsf\Pi f}{\mathrm L^2(\d x\,\d\mu)}^2}{\nrm f{\mathrm L^1(\d x\,\d v)}^2}\le\mathsf A\,,
\]
then
\[
\nrm{\Lambda^\frac12\,\mathsf\Pi\Ff}{\mathrm L^2(\d x\,\d\mu)}^2\ge\tfrac{\omega_d}{2\,d}\,\nrm f{\mathrm L^1(\d x\,\d\mu)}^2\,\Phi\(\frac{\nrm{\mathsf\Pi f}{\mathrm L^2(\d x\,\d\mu)}^2}{\nrm f{\mathrm L^1(\d x\,\d v)}^2}\).
\]
\end{lemma}
%---------------------------------------------------------------------
\begin{proof}
As in the case $\gamma\neq2+\beta$, we use
\begin{align*}
\nrm{\mathsf\Pi f}{\mathrm L^2(\d\xi\,\d\mu)}^2&=\nrm{\rho_f}{\mathrm L^2(\d x)}^2=\int_{|\xi|<R}|\widehat{\rho_f}|^2\,\d\xi+\int_{|\xi|\ge R}|\widehat{\rho_f}|^2\,\d\xi\\
&\le\tfrac{\omega_d}d\,R^d\,\nrm{\widehat{\rho_f}}{\mathrm L^\infty(\d\xi)}^2+\frac1{h\(R\)}\int_{|\xi|\ge R} \Lambda\,|\widehat{\rho_f}|^2\,\d\xi\\
&\le\tfrac{\omega_d}d\,R^d\,\nrm f{\mathrm L^1(\d x\,\d\mu)}^2+\frac1{h\(R\)}\,\nrm{\Lambda^\frac12\,\mathsf\Pi\Ff}{\mathrm L^2(\d x\,\d\mu)}^2
\end{align*}
for some $R>0$, small enough. The last inequality can be written as
\[
X\le R^d\,a+\frac b{h(R)}
\]
with $a=\tfrac{\omega_d}d\,\nrm f{\mathrm L^1(\d x\,\d\mu)}^2$, $b=\nrm{\Lambda^\frac12\,\mathsf\Pi\Ff}{\mathrm L^2(\d x\,\d\mu)}^2$ and $X=\nrm{\mathsf\Pi f}{\mathrm L^2(\d\xi\,\d\mu)}^2$. There is a unique $R>0$, small, such that $R^{d+2}\,|\log R|\sim R^d\,h(R)=b/a$ if $b/a$ is small enough, from which we deduce that $X\le2\,a\,R^d$, \emph{i.e},
\[
\frac ba\gtrsim\Phi\(\frac X{2\,a}\)\quad\mbox{where}\quad\Phi(x):=\frac 1d\,x^{1+\frac2d} \left|\log x\right|\,.
\]
The conclusion holds for some $\mathsf A<R^d\,h(R)$ with $R=1/\sqrt e$, whose detailed expression is inessential.\Qed\end{proof}

%%%%%%%%%%%%%%%%%%%%%%%%%%%%%%%%%%%%%%%%%%%%%%%%%%%%%%%%%%%%%%%%%%%%%%
\subsection{An extension of \texorpdfstring{Corollary~\ref{Cor:fractionalNash2} when $\gamma\le \beta$}{beta+gamma le0}}\label{Sec:Extbeta+gamma<0}

We do not have a good control of $\vertiii{(1-\mathsf\Pi)f}_{- \beta}^2$ when $ \beta\ge\,\gamma$, but we claim that the issue can be solved if we consider $\nrm{(1-\mathsf\Pi)f}{\mathrm L^2(\d x\,\bangle v^\eta\d\mu)}^2$.
%---------------------------------------------------------------------
\begin{corollary}\label{Cor:fractionalNash2ext} Let $\gamma\le \beta$ and $\eta\in(-\,\gamma,0)$. Under Assumption~{\rm (H)}, if $f$ is a solution of~\eqref{eq:main}, then for any $t\ge0$,
\[
\rint\mathsf R_\xi[\Ff]\,\d\xi\;\gtrsim\;\nrm{\mathsf\Pi f}{\mathrm L^2(\d x\,\d\mu)}^{2\,(1+\frac\alpha d)}-\nrm{(1-\mathsf\Pi)f}{\mathrm L^2(\d x\,\bangle v^\eta\d\mu)}^2\,.
\]\end{corollary}
%---------------------------------------------------------------------

%%%%%%%%%%%%%%%%%%%%%%%%%%%%%%%%%%%%%%%%%%%%%%%%%%%%%%%%%%%%%%%%%%%%%%
%%%%%%%%%%%%%%%%%%%%%%%%%%%%%%%%%%%%%%%%%%%%%%%%%%%%%%%%%%%%%%%%%%%%%%
\section{Completion of the proofs and extension}\label{Sec:MoreProofs}

%%%%%%%%%%%%%%%%%%%%%%%%%%%%%%%%%%%%%%%%%%%%%%%%%%%%%%%%%%%%%%%%%%%%%%
\subsection{The case \texorpdfstring{$\gamma\le \beta$}{beta+gamma le0}}\label{Sec:beta+gamma<0}

Let us define
\[
X_\zeta:=\riint|(1-\mathsf\Pi)f|^2 \bangle v^\zeta\d x\,\d\mu\quad\mbox{and}\quad Y:=\riint|\mathsf\Pi f|^2\,\d x\,\d\mu\,.
\]
With this notation, Inequality~\eqref{Z} in Corollary~\ref{Cor:WPIIext} can be written as
\be{eta1}
\forall\,r>0\,,\quad -\irdx{\bangle{f,\mathsf Lf}}\ge r\,X_\eta-(1-\theta)\,\theta^\frac\theta{1-\theta}\,r^\frac1{1-\theta}\,X_k
\ee
with $\theta=\frac{k-\eta}{k+\beta}$, $- \beta\le-\,\gamma<\eta<0<k<\gamma$, while Corollary~\ref{Cor:fractionalNash2ext} simply means
\be{eta2}
\rint\mathsf R_\xi[\Ff]\,\d\xi\;\gtrsim\;Y^{1+\frac\alpha d}-X_\eta\,.
\ee
Let us consider
\[
\mathcal H:=X_0+Y+\delta(t)\,\re\(\irdx{\bangle{\mathsf Af,f}}\)\mbox{with}\quad\delta(t)=\delta_0\,(1+\varepsilon\,t)^{-\mathsf a}\,,
\]
for some constant numbers $\mathsf a\in(0,1)$, $\delta_0>0$ and $\varepsilon>0$, to be chosen. The major difference with the case $\gamma>\beta$ considered in Section~\ref{Sec:Proofs} is that we allow $\delta$ to depend on $t$ and that we shall actually make an explicit choice of this dependence.

We know that
\[
\big(1-\tfrac{\delta}2\big)\,(X_0+Y)\le\mathcal H\le\big(1+\tfrac{\delta}2\big)\,(X_0+Y)
\]
and compute
\[
-\frac{\d\mathcal H}{\d t}=-\,2\,\irdxmu{f\,\mathsf L f}+\delta(t)\rint\mathsf R_\xi[\Ff]\,\d\xi+\delta'(t)\,\re\(\irdx{\bangle{\mathsf Af,f}}\).
\]
Using~\eqref{eta1} and~\eqref{eta2}, we get the estimate
\[
-\frac{\d\mathcal H}{\d t}\gtrsim\delta\,Y^{1+\frac\alpha d}-\delta\,X_\eta+r\,X_\eta-r^\frac{k+ \beta}{\eta+\beta}-\tfrac{\delta\,\varepsilon}{1+\varepsilon\,t}\,\mathcal H\,.
\]
We recall that $\gtrsim$ means that the inequality holds up to a positive, finite constant, which changes from line to line. Next we choose $r=2\,\delta$, $\delta_0>0$ and $\varepsilon>0$ small enough so that the above right-hand side of the inequality is positive. However, we shall still do some further reductions before fixing the values of $\delta_0$ and $\varepsilon$. The decay rate of~$\mathcal H$ is governed by
\[
- \frac{\d\mathcal H}{\d t} \gtrsim\delta\,Y^{1+\frac\alpha d}+\delta\,X_\eta-\delta^\frac{k+ \beta}{\eta+\beta}-\tfrac{\delta\,\varepsilon}{1+\varepsilon\,t}\,\mathcal H\,,
\]
with a positive right-hand side at $t=0$. Using H\"older's inequality
\[
X_0\le X_\eta^\frac k{k-\eta}\,X_k^\frac\eta{\eta-k}
\]
and the fact that $X_k$ is uniformly bounded in $t$ by a positive constant depending only on the initial datum, we obtain
\[
-\frac{\d\mathcal H}{\d t}\gtrsim\delta\,Y^{1+\frac\alpha d}+\delta\,X_0^{1-\frac\eta k}-\delta^\frac{k+\beta}{\eta+\beta}-\tfrac{\delta\,\varepsilon}{1+\varepsilon\,t}\,\mathcal H\,,
\]
and we can still assume that the inequality has a positive right-hand side at $t=0$ without loss of generality. It is now clear that
\[
-\frac{\d\mathcal H}{\d t}\gtrsim\delta\(\mathcal H^{1+\kappa}-\delta^\frac{k-\eta}{\eta+\beta}-\tfrac\varepsilon{1+\varepsilon\,t}\,\mathcal H\)\quad\mbox{with}\quad\kappa=\max\left\{\frac\alpha d,-\frac\eta k\right\}\,.
\]
Up to a multiplication by a constant, we can actually fix the mutiplicative constant to a given value $\tau>0$ that will be chosen below (with the corresponding redefinition of $\varepsilon$ and $\delta_0$) so that the differential inequality is
\[
-\frac{\d\mathcal H}{\d t}\ge\tau\,\delta\(3\,\mathcal H^{1+\kappa}-\delta^\frac{k-\eta}{\eta+ \beta}-\tfrac\varepsilon{1+\varepsilon\,t}\,\mathcal H\).
\]
Now, let us fix $\delta_0>0$ and $\varepsilon>0$ small enough so that
\[
\delta_0^\frac{k-\eta}{\eta+\beta}+\varepsilon\,\mathcal H_0\le\mathcal H_0^{1+\kappa}
\]
with $\mathcal H_0=\mathcal H(t=0)$. We have to check that this condition is stable under the evolution, that is,
\be{Stab}
\forall\,t\ge0\,,\quad\delta(t)^\frac{k-\eta}{\eta+\beta}\le\mathcal H(t)^{1+\kappa}\quad\mbox{and}\quad\varepsilon\,\mathcal H(t)\le\mathcal H(t)^{1+\kappa}\,.
\ee

Keeping track of the coefficients is paid by unnecessary complications, so that we are going to make some simplifying assumptions, in order to emphasize the key idea of the estimate. Up to a change of variable $t\mapsto\varepsilon\,t$, we can choose $\varepsilon=1$ and also take $\delta_0=1$ and $\mathcal H_0=1$ without loss of generality, so that, in particular,
\[
\forall\,t\ge0\,,\quad\delta(t)=(1+t)^{-\mathsf a}\,.
\]
As a result, let us consider the differential inequality
\[
-\frac{\d\mathcal H}{\d t}\ge\tau\,\delta\(3\,\mathcal H^{1+\kappa}-\delta^\frac{k-\eta}{\eta+ \beta}-\tfrac{\mathcal H}{1+t}\).
\]
We aim at showing that
\be{barrier}
\mathcal H(t)\le\overline{\mathcal H}(t):=(1+t)^{-\,\tau}
\ee
where $\overline{\mathcal H}$ solves
\[
\frac{\d\overline{\mathcal H}'}{\d t}=-\,\tau\,\delta\,\overline{\mathcal H}^{1+\kappa}\quad\mbox{with}\quad\tau=\frac{1-\mathsf a}\kappa\,.
\]
As a consequence of
\[
\delta^\frac{k-\eta}{\eta+\beta}=(1+t)^{-\mathsf a\,\frac{k-\eta}{\eta+\beta}}\quad\mbox{and}\quad\tfrac{\overline{\mathcal H}}{1+t}=(1+t)^{-\frac{1-\mathsf a}\kappa-1}\,,
\]
we learn that
\[
\forall\,t\ge0\,,\quad\delta(t)^\frac{k-\eta}{\eta+\beta}\le\overline{\mathcal H}(t)^{1+\kappa}\quad\mbox{and}\quad\tfrac{\overline{\mathcal H}}{1+t}\le\overline{\mathcal H}(t)^{1+\kappa}
\]
under the condition that
\be{Rate}
-\mathsf a\,\frac{k-\eta}{\eta+\beta}\le-\,(1-\mathsf a)\,\frac{1+\kappa}\kappa\quad\mbox{and}\quad\mathsf a\ge0\,.
\ee
Since $-\,\frac{\d\mathcal H'}{\d t}\le-\,\frac{\d\overline{\mathcal H}'}{\d t}$ if $\mathcal H(t)=\overline{\mathcal H}(t)$, it is then clear that~\eqref{Stab} holds and $\mathcal H(t)\le\overline{\mathcal H}(t)$ for any $t\ge0$. In other words, $\overline{\mathcal H}$ is a barrier function and~\eqref{barrier} holds for any $t\ge0$. The result is also true for the generic case.

With the choice $\mathsf a=( \beta+\eta)/\beta$, Condition~\eqref{Rate} is satisfied if $- \beta\le\eta\le0\le k$ and $\kappa=|\eta|/k$, with $\tau=k/|\beta|$ if $d\ge2$ because
\[
\alpha=\frac{\gamma+\beta}{1+ \beta}\le\max\left\{1,\frac{2\,\gamma}{1+\gamma}\right\}
\]
and because $\kappa=|\eta|/k>\alpha/d$ for an appropriate choice of $\eta\in(-\,\gamma,0)$ and $k\in(0,|\eta|\,d/\alpha)$. The same argument applies if $d=1$ and $\gamma\le1$.

If $d=1$, in the range $1\le\gamma\le| \beta|$, we have $\alpha>1$ and distinguish two cases.\\
$\bullet$ Either $\kappa=\alpha>|\eta|/k$: with $|\eta|/\alpha<k<\gamma$ and $\eta>-\,\gamma$, we find that $\tau=\tau_\star(\eta,k)$ and $a=a_\star(\eta,k)$ where
\[
\tau_\star(\eta,k):=\frac{k-\eta}{\alpha\,(k+\beta)+\eta+ \beta}\quad\mbox{and}\quad a_\star(\eta,k):=\frac{(1+\alpha)\,(\eta+\beta)}{\alpha\,(k+\beta)+\eta+\beta}\,.
\]
Using $\frac{\partial\tau_\star}{\partial k}>0$ and $\frac{\partial\tau_\star}{\partial\eta}<0$, the largest admissible value of $\tau_\star$ is achieved by
\begin{align*}
\lim_{(\eta,k)\to(-\,\gamma,\gamma)}\tau_\star(\eta,k)&=\frac{2\,\gamma}{\alpha\,(\gamma+ \beta)+|\gamma - \beta|}\,.
\end{align*}
$\bullet$ Or $\alpha\le|\eta|/k=\kappa$: in that case, we can take $\mathsf a=( \beta + \eta)/ \beta$, Condition~\eqref{Rate} is satisfied if $-\,\gamma\le\eta\le0\le k\le|\eta|/\alpha<\gamma/\alpha$, $\tau=k/|\beta|$ and $\kappa=|\eta|/k$. This is possible as soon as $k<\gamma/\alpha$ since this condition is then verified if we take any $\eta\in(-\,\gamma,-\,k\,\alpha)$.

This completes the proof of Theorems~\ref{th:main2} and~\ref{th:main1} with $\tau=k/|\beta|$, except if $d=1$, $1\le\gamma\le|\beta|$ and $k\in[\gamma/\alpha,\gamma)$, where the rate can be chosen arbitrarily close to $\tau_\star(-\,\gamma,k)$.

%%%%%%%%%%%%%%%%%%%%%%%%%%%%%%%%%%%%%%%%%%%%%%%%%%%%%%%%%%%%%%%%%%%%%%
\subsection{The case of a flat torus}\label{Sec:torus}

As in~\cite{bouin_hypocoercivity_2017}, the case of the flat $d$-dimensional torus~$\mathbb T^d$ (with position $x\in\mathbb T^d$ and velocity $v\in\R^d$) follows from our method without additional efforts. In that case, Equation~\eqref{eq:main} admits a global equilibrium given by $f_\infty=\rho_\infty\,\M$ with $\rho_\infty=\frac1{\left|\mathbb T^d\right|} \iint_{\mathbb T^d\times\R^d} f^{\mathrm{in}}\,\d x\,\d v$, and the rate of convergence to the equilibrium is just given by the microscopic dynamics
\begin{align*}
\nrm{f-f_\infty}{\mathrm L^2(\d x\,\d\mu)}&\lesssim e^{-\lambda\,t}\,\nrm{f^{\mathrm{in}}-f_\infty}{\mathrm L^2(\d x\,\d\mu)}&\quad\mbox{if}\quad \beta \in (- \gamma,0]\,,\\
\nrm{f-f_\infty}{\mathrm L^2(\d x\,\d\mu)}&\lesssim(1+t)^{-k/\beta}\,\nrm{f^{\mathrm{in}}-f_\infty}{\mathrm L^2(\bangle v^k\d x\,\d\mu)}&\quad\mbox{if}\quad \beta>0\,,
\end{align*}
with $k\in(0,\gamma)$. In particular, if $f=f(t,v)$ does not depend on $x$, then~\eqref{eq:main} is reduced to the homogeneous equation $\partial_tf=\mathsf Lf$ and we recover the rate of convergence of $f$ to $\M$ in the norm $\mathrm L^2(\d\mu)$, as in Section~\ref{Sec:intro}. This is coherent with the results in~\cite{bakry_rate_2008,wang_simple_2014,wang_functional_2015,chen_weighted_2017,lafleche_fractional_2020,2019arXiv191111535A}. Moreover, we point out that our result is a little bit stronger than some of those results, because it relies on a finite $\nrm{f^{\mathrm{in}}}{\mathrm L^2(\bangle v^k\d x\,\d\mu)}$ norm for the initial condition, which is a weaker condition than the usual boundedness condition on $\nrm{f^{\mathrm{in}}\,\M^{-1}}{\mathrm L^\infty(\d x\,\d v)}$, or $\mathrm H^1$-type estimates as in~\cite[Section~6]{2019arXiv191111535A}, where $\beta=0$ in Case~$\mathsf L=\mathsf L_2$. Remark however that weighted $\mathrm L^2$ norms already appear in the homogeneous case in~\cite{kavian_fokker-planck_2015,lafleche_fractional_2020}.

%%%%%%%%%%%%%%%%%%%%%%%%%%%%%%%%%%%%%%%%%%%%%%%%%%%%%%%%%%%%%%%%%%%%%%
\subsection{About rates in the diffusion limits}\label{Sec:difflim}

We shall end this paper with a comment about the stability of the rates in the diffusive scaling. In the range of parameters for which we prove decay at rate $t^{-\frac{d}{\alpha}}$, our previous computations give that the rate is uniform in the rescaling $t \to \frac{t}{\varepsilon^\alpha}$ and $x \to \frac{x}{\varepsilon}$. When the rate is given by $t^{-\frac{k}{\beta}}$, the way that the rate degenerates into $t^{-\frac{d}{\alpha}}$ (the rate of the macroscopic limit, see e.g. \cite{bouin2020quantitative}) is a bit more intricate, and we leave this issue for future work. 

%%%%%%%%%%%%%%%%%%%%%%%%%%%%%%%%%%%%%%%%%%%%%%%%%%%%%%%%%%%%%%%%%%%%%%
%%%%%%%%%%%%%%%%%%%%%%%%%%%%%%%%%%%%%%%%%%%%%%%%%%%%%%%%%%%%%%%%%%%%%%
\section*{Acknowledgments}
{\small This work has been supported by the Project EFI (ANR-17-CE40-0030) of the French National Research Agency (ANR). The authors are deeply indebted to Christian Schmeiser, who was at the origin of the research project and participated to the preliminary discussions while he was visiting as the holder of the \emph{Chaire d'excellence} of the \emph{Fondation Sciences Math\'ematiques de Paris}, also supported by \emph{Paris Sciences et Lettres}. The project is also part of the Amadeus project \emph{Hypocoercivity} no.~39453PH.}\\
{\sl\scriptsize\copyright~2021 by the authors. This paper may be reproduced, in its entirety, for non-commercial purposes.}
%%%%%%%%%%%%%%%%%%%%%%%%%%%%%%%%%%%%%%%%%%%%%%%%%%%%%%%%%%%%%%%%%%%%%%

\appendix
\section{Steady states and force field for the fractional Laplacian with drift}\label{Sec:Ebound}

This appendix is devoted to the Case~$\mathsf L=\mathsf L_3$ of the collision operator $\mathsf L$, that is, to $\mathsf L_3 f=\Delta_v^{\sigma/2}f+\nabla_v\cdot\(E\,f\)$. Our goal here is to prove that the \emph{collision frequency} $\nu(v)$ behaves like $|v|^{- \beta}$ with $\beta= \sigma -\gamma$ as $|v|\to+\infty$, as claimed in Section~\ref{Sec:intro}. By Definition~\eqref{eq:def_E} of the force field $E$, we know that
\[
\nabla_v\cdot(E\,\M)=-\Delta_v^{\sigma/2} \M=-\,\nabla_v\cdot\(\nabla_v(-\Delta_v)^\frac{\sigma-2}2\,\M\),
\]
and this implies that, up to an additive constant,
\[
E\,\M=-\,\nabla_v(-\Delta_v)^\frac{\sigma-2}2\,\M=-\,\nabla_v\(\frac{C_{d,\sigma}}{|v|^{d+\sigma-2}}*\frac{c_\gamma}{\bangle v^{d+\gamma}}\)
\]
where $c_\gamma$ and $C_{d,\sigma}$ are given respectively by~\eqref{eq:hypfattailed} and~\eqref{Garofalo}.
%---------------------------------------------------------------------
\begin{proposition}\label{prop:estE} Assume that $\gamma>0$, $\sigma\in(0,2)$ and let $ \beta=\sigma - \gamma$. There is a positive function $G\in\mathrm L^\infty(\R^d)$ with $1/G\in\mathrm L^\infty(B_0^c(1))$ such that $E$ is given by
\[
\forall\,v\in\R^d\,,\quad E(v)=G(v) \bangle v^{- \beta} v\,.
\]
\end{proposition}
%---------------------------------------------------------------------
\begin{proof}
Let $u(v)=-\nabla_v\,\big(\frac1{|v|^{d+\sigma-2}}*\frac1{\bangle v^{d+\gamma}}\big)(v)$ so that $E(v)=C_{d,\sigma} \bangle v^{d+\gamma} u(v)$. Since
\[ 
u(v)=(d+\gamma)\(\frac1{|v|^{d+\sigma-2}}*\frac v{\bangle v^{d+\gamma+2}}\)
\]
where $\bangle v^{-(d+\gamma+2)} v\in C^\infty(\R^d)\cap\mathrm L^1(\d v)$, and $\sigma<2$, one has $u\in C^1_{\mathrm{loc}}(\R^d)$ and $u(0)=0$ which proves the result in $B_1(0)$. We look for an estimate of $u(v)\cdot v$ from above and below on $B_0^c(1)$. Notice that $u$ can also be written as
\be{eq:estim_E_u}
u(v)=\(d+\sigma-2\)\(\frac v{|v|^{d+\sigma}}*\frac1{\bangle v^{d+\gamma}}\).
\ee
Depending on the integrability at infinity of $v/|v|^{d+\sigma}$, that is, whether $\sigma\in(0,1)$ or not, we have to distinguish two cases.
\\[4pt]
\noindent $\bullet$ \textbf{Case $\sigma\in(0,1)$.} Using~\eqref{eq:estim_E_u}, we have the estimates
\begin{align*}
\left|\int_{|w|\ge\bangle v/2}\frac w{|w|^{d+\sigma}}\frac{\d w}{\bangle{w-v}^{d+\gamma}}\right|&\le\frac{2^{d+\sigma-1}}{\bangle v^{d+\sigma-1}}\rint\frac{\d w}{\bangle w^{d+\gamma}}\,,\\
\left|\int_{|w|<\bangle v/2}\frac w{|w|^{d+\sigma}}\frac{\d w}{\bangle{w-v}^{d+\gamma}}\right|&\le\(\int_{|w|<\bangle v/2}\frac{\d w}{|w|^{d+\sigma-1}}\)\frac{2^{d+\sigma-1}}{|v|^{d+\gamma}}\\
&\hspace*{2cm}\le\frac{2^{d+\sigma-1}\,\omega_d}{(1-\sigma)\,|v|^{d+\gamma+\sigma-1}}\,,
\end{align*}
and obtain
\[
\forall\,v\in\R^d\,,\quad|u(v)\cdot v|\le|u(v)|\,|v|\lesssim|v|^{-(d+\sigma-2)}\,.
\]

To get a bound from below on $u(v)\cdot v$, we cut the integral in two pieces and use the fact that $|v|>1$ and $|w-v|<1/2$ implies $w\cdot v>0$. First
\begin{multline*}
\int_{\substack{|w-v|>1/2\\ |w+v|>1/2}}\frac{w\cdot v}{|w|^{d+\sigma}}\frac{\d w}{\bangle{w-v}^{d+\gamma}}=\(\int_{\substack{|w-v|>1/2\\ w\cdot v>0}}\kern-3pt+\kern-3pt\int_{\substack{|w+v|>1/2\\ w\cdot v<0}}\)\kern-3pt\frac{w\cdot v}{|w|^{d+\sigma}}\frac{\d w}{\bangle{w-v}^{d+\gamma}}\\
=\int_{\substack{|w-v|>1/2\\ w\cdot v>0}}\(\frac1{\bangle{w-v}^{d+\gamma}}-\frac1{\bangle{w+v}^{d+\gamma}}\)\frac{w\cdot v}{|w|^{d+\sigma}}\,\d w\,,
\end{multline*}
which is positive since $\bangle{w+v}^2-\bangle{w-v}^2=2\,w\cdot v\ge 0$. The remaining terms are dealt with as follows
\begin{multline*}
\int_{\substack{|w-v|\le1/2\\[2pt]\mbox{\scriptsize or}\\[2pt] |w+v|\le1/2}}\frac{w\cdot v}{|w|^{d+\sigma}}\frac{\d w}{\bangle{w-v}^{d+\gamma}}\\
=\int_{|w-v|<\frac12}\(\frac1{\bangle{w-v}^{d+\gamma}}-\frac1{\bangle{w+v}^{d+\gamma}}\)\frac{w\cdot v}{|w|^{d+\sigma}}\,\d w\\
\ge\((4/5)^{d+\gamma}-(2/5)^{d+\gamma}\)\int_{|w-v|<\frac12}\frac{w\cdot v}{|w|^{d+\sigma}}\,\d w\,,
\end{multline*}
since $|w+v|\ge2\,|v|-|w-v|\ge\frac32$. Finally, if $|v|>1$ and $|w-v|<\frac12$, we get
\begin{align*}
2\,w\cdot v&=|v|^2+|w|^2-|w-v|^2\ge|v|^2-\frac12\ge\frac{|v|^2}2\,,\\
|w|&\le|v|+|w-v|\le2\,|v|\,,
\end{align*}
so that
\begin{align*}
\int_{|w-v|<\frac12}\frac{w\cdot v}{|w|^{d+\sigma}}\,\d w\ge\frac{|B_0(1/2)|}{2^{d+\sigma+2}}\frac1{|v|^{d+\sigma-2}}\,.
\end{align*}
This implies $u(v)\cdot v\ge C\,|v|^{-(d+\sigma-2)}$ for some $C>0$. Since $u$ is radial, we proved that
\begin{align*}
u(v)=G(v)\,\frac v{|v|^{d+\sigma}}
\end{align*}
where $G\in\mathrm L^\infty(\R^d)$ and $G^{-1}\in\mathrm L^\infty(B_0^c(1))$ and the conclusion holds with $\beta=\sigma - \gamma$.

\medskip\noindent $\bullet$ \textbf{Case $\sigma\in[1,2)$}. The gradient of $v\mapsto|v|^{2-d-\sigma}$ is a distribution of order $1$ that can be defined as a \emph{principal value}. Indeed, in the sense of distributions, for any $\varphi\in\mathcal D(\R^d)$, we have
\begin{align*}
\bangle{\nabla_v|v|^{2-d-\sigma},\varphi}_{\mathcal D',\mathcal D}&=-\rint\frac{\nabla_v\varphi(v)}{|v|^{d+\sigma-2}}\,\d v=-\rint\frac{\nabla_v(\varphi(v)-\varphi(0))}{|v|^{d+\sigma-2}}\,\d v\\
&=(d+\sigma-2)\rint\frac v{|v|^{d+\sigma}}\(\varphi(v)-\varphi(0)\)\d v\\
&=: (d+\sigma-2) \bangle{\pv\!\(\frac v{|v|^{d+\sigma}}\)\!,\varphi}_{\mathcal D',\mathcal D}\,.
\end{align*}
Identity~\eqref{eq:estim_E_u} is replaced by
\[
\frac{u(v)}{d+\sigma-2}=\pv\!\(\frac v{|v|^{d+\sigma}}\)*\frac1{\bangle v^{d+\gamma}}=\rint\frac w{|w|^{d+\sigma}}\(\frac1{\bangle{v-w}^{d+\gamma}}-\frac1{\bangle v^{d+\gamma}}\)\d w\,,
\]
so that, after computations like the ones in the proof of Lemma~\ref{lem:estim_lap},
\[
\frac{|u(v)|}{d+\sigma-2}\le\rint\frac1{|w-v|^{d+\sigma-1}}\left|\frac1{\bangle w^{d+\gamma}}-\frac1{\bangle v^{d+\gamma}}\right|\d w\lesssim\frac1{\bangle v^{d+\sigma-2}}\,.
\]
Now estimate $u(v)\cdot v$ by
\begin{multline*}
\int_{|w|\ge\frac12}\frac{w\cdot v}{|w|^{d+\sigma}}\(\frac1{\bangle{v-w}^{d+\gamma}}-\frac1{\bangle v^{d+\gamma}}\)\d w\\=\int_{|w|\ge\frac12}\frac{w\cdot v}{|w|^{d+\sigma}}\frac1{\bangle{v-w}^{d+\gamma}}\,\d w\gtrsim\frac1{\bangle v^{d+\sigma-2}}\,.
\end{multline*}
and
\begin{multline*}
\left\vert\int_{|w|<\frac12}\frac{w\cdot v}{|w|^{d+\sigma}}\(\frac1{\bangle{v-w}^{d+\gamma}}-\frac1{\bangle v^{d+\gamma}}\)\d w \right\vert\\
\le\sup_{w\in B_v(1/2)}\frac{(d+\gamma)\,|v|}{\bangle w^{d+\gamma+1}}\int_{|w|<\frac12}\frac{\d w}{|w|^{d+\sigma-2}}\lesssim\frac1{\bangle v^{d+\gamma}}\,.
\end{multline*}
The result follows from the fact that $d+\gamma>d>d+\sigma-2$.
\Qed\end{proof}

%%%%%%%%%%%%%%%%%%%%%%%%%%%%%%%%%%%%%%%%%%%%%%%%%%%%%%%%%%%%%%%%%%%%%%
\bibliographystyle{acm}
\bibliography{BDL}
%%%%%%%%%%%%%%%%%%%%%%%%%%%%%%%%%%%%%%%%%%%%%%%%%%%%%%%%%%%%%%%%%%%%%%
\newpage\tableofcontents
%%%%%%%%%%%%%%%%%%%%%%%%%%%%%%%%%%%%%%%%%%%%%%%%%%%%%%%%%%%%%%%%%%%%%%
\end{document}